%% file: main.tex
\documentclass{article}

\usepackage[
    linktocpage=true,
    colorlinks=true,
    linkcolor=cyan,   % page numbers in ToC
    citecolor=cyan,   % bibliography links
    urlcolor=cyan     % URLs
]{hyperref}

\usepackage{graphicx} % Required for inserting images
\usepackage{amsmath}
\usepackage{hyperref}
\usepackage{upgreek}
\usepackage{xcolor}

\usepackage[toc,title,page]{appendix}
\usepackage{cleveref}
\usepackage[caption=false]{subfig}
\usepackage[margin=3cm]{geometry}
\usepackage{todonotes}
\usepackage{amsthm}
\usepackage{amssymb}
\usepackage{enumitem}
\usepackage{relsize}
\setlist[enumerate]{nosep}

\usepackage[backend=biber,
            style=numeric-comp,
            sortcites=true,
            giveninits=true, % Use initials in bibliography
            useprefix=true,
            sorting=nyt,
            maxbibnames=9]{biblatex}
\DeclareFieldFormat{titlecase}{#1}
\newtheorem{prop}{Proposition}
\newtheorem{thm}{Theorem}
\newtheorem{definition}{Definition}
\newtheorem{remark}{Remark}
\newtheorem{lem}{Lemma}

\numberwithin{prop}{section}
\numberwithin{lem}{section}
\numberwithin{thm}{section}
\numberwithin{definition}{section}
\numberwithin{cor}{section}
\numberwithin{remark}{section}

\input{macro}

\title{Layered Dissipative Neural Fields:\\
Bounding the Dimension of the Population Activity Manifold with Infinitely Many Neurons}

\author{Matthias Rakotomalala\thanks{Technical University of Munich,
Chair of Analysis and Modelling,  Boltzmannstr. 3,
85748 Garching b. M\"unchen, Germany.}, Johannes Zimmer\footnotemark[1]}

\begin{document}

\maketitle

\begin{abstract}
    We consider a class of neural field models describing the macroscopic activity of an infinite population of neurons organized into finitely many stacked layers, written as a system of semi-linear higher-order parabolic equations. For this class of biologically inspired equations, we show that the higher-order dissipation, modeling \textit{gap junction} activity regulation, induces a spectral gap ensuring that the associated semigroup possesses an \textit{inertial manifold}, that is, a finite-dimensional manifold attracting all orbits of the infinite-dimensional system of neural activity. Our main result is an upper bound on the dimension of this manifold, giving its explicit scaling in the dissipation order and strength, the leaking rate, and the norm of the connection operator. We interpret this as a step toward the mathematical modeling of the empirically motivated concept of \textit{neural manifold}, namely the idea that the collective activity of a large neural population lies on a low-dimensional manifold. For this class of equations, we further prove universal approximation properties, namely stationary pattern expressivity in terms of the input and input-dependent dynamical expressivity over finite time horizons. Finally, we provide sufficient conditions for the existence of traveling waves and nontrivial stationary solutions.
    %In this paper, we consider a generalized dissipative neural field model, describing the macroscopic activity of an infinite population of neurons organized into finitely many stacked layers. The model we consider is a system of semi-linear higher-order parabolic equations. For this class of biologically inspired equations, we give sufficient conditions for the associated semigroup to possess an inertial manifold; that is, a finite-dimensional manifold that attracts all orbits of the infinite-dimensional system of neural activity. This result is interpreted as a step toward the mathematical modeling of the empirically motivated concept of \textit{neural manifold}, namely the idea that the collective activity of a large neural population lies on a low-dimensional manifold. Furthermore, we provide a explicit upper-bound on the dimension that depending on the norm of the connection operator. For this class of equations, we prove universal approximation properties and establish nontrivial qualitative behavior of the equation. Specifically, the existence of nontrivial stationary solutions and traveling waves is shown.
\end{abstract}


\tableofcontents

\section{Introduction}

\input{Intro}

\section{Notation and Preliminaries}
\label{subsec:FunctionFramework}
\input{preliminaries}

\section{Derivation of the Layered Dissipative Neural Field Equation}
\label{sec:Derivation}
\input{Derivation}

\section{Statement of the Results}
\label{sec:MainRestult}
\input{Mainresults}

\section{Generalized Layered Dissipative Neural Fields}
\label{sec:GeneralizedDNF}
\input{GeneralizedLDNF}

\section{Proofs of the Results}
\label{sec:Proofs}
\input{ProofsAbstractResult}

\section*{References}
\addcontentsline{toc}{section}{References}
\printbibliography[heading=none]
    
\end{document}

%% file: macro.tex
\newcommand{\dR}{\mathbb{R}}
\newcommand{\dZ}{\mathbb{Z}}
\newcommand{\dN}{\mathbb{N}}
\newcommand{\dC}{\mathbb{C}}
\newcommand{\Torus}{\mathbb{T}}

\newcommand{\dH}{\mathbb{H}}
\newcommand{\dL}{\mathbb{L}}

\newcommand{\dd}{\mathrm{d}}
\DeclareMathOperator{\dist}{dist}
\DeclareMathOperator{\osc}{osc}
\DeclareMathOperator{\coker}{coker}

\DeclareMathOperator{\argmax}{argmax}
\DeclareMathOperator{\Rank}{Rank}
\DeclareMathOperator{\supp}{supp}
\newcommand{\Indic}{\mathbf{1}}

\newcommand{\Real}{\mathrm{Re}}
\newcommand{\Imag}{\mathrm{Im}}

\newcommand{\In}{\mathrm{in}}
\newcommand{\I}{\mathcal{I}}

\renewcommand{\P}{\mathrm{P}}
\newcommand{\Q}{\mathrm{Q}}

\theoremstyle{definition}
\newtheorem{assumption}{Assumption}

\newcommand{\vnu}{\boldsymbol{\upnu}}
\newcommand{\vn}{\mathbf{n}}
\newcommand{\vtau}{\boldsymbol{\uptau}}

\newcommand{\vs}{\mathbf{s}}
\newcommand{\vr}{\mathbf{r}}

\newcommand{\Trps}{\mathsf{T}}
\DeclareMathOperator{\diag}{diag}

\DeclareMathOperator{\Span}{span}

\newcommand{\W}{\mathrm{W}}

%% file: Intro.tex
%%%%%%%%%%%%%%%%%%%%%%%%%%%%%%%%%%%%%%%%%%%%%%%%%%%%%%%%%%%%%%%%%%
%      - paragraph introductif du domain
%      - manifold hypothesis
%      - le model
%      - revue de litterature
%      - plan du papier
%
%

%           RECHERCHES
%
%
%       https://www.thetransmitter.org/neural-dynamics/neural-manifolds-latest-buzzword-or-pathway-to-understand-the-brain/
%
% 

% https://drive.google.com/file/d/1XGIFI4Md62CUr6rEe8ESVQZVf9PXwHrV/view
%   In neuroscience, empirical evidence from across many brain
% regions, species and behaviors (see review sections below) shows that
% the latent dynamics of neural populations consistently explore only a
% subset of possible states. T
%

    % Neural manifolds are likely to emerge from the coordinated activity of interacting neurons within (and probably across) brain regions
    % (Fig. 1a).

%LINK with Recurrent neural network BOX3

%
% Wikipedia
        
        % How is information in the brain processed by the collective dynamics of large neuronal circuits?
        % What level of simplification is suitable for a description of information processing in the brain?
        % What is the neural code?
        
        % % 

%complex coordination of neural population
%
%
%
%.  CYTOARCHITECHTURE MULTIPLE LAYER OR 
%
%
%

%%%%%%%%%%%%%%%%%%%%%%%%%%%%%%%%%%%%%%%%%%%%%%%%%%%%%%%%%%%%%%%%%%

Two of the most challenging open questions in computational neuroscience are \textit{how information is encoded in the brain}, and \textit{how computation emerges from the coordinated dynamics of large neuronal populations}. Recent studies approach these questions through the prism of \textit{neural manifolds}: empirical data show that the activities of large neuronal populations are constrained to low-dimensional manifolds of population-wide patterns of activity. In this paper, through mathematical modeling and analysis, we present a potential mechanism providing this collective behavior. Namely, we consider a formal mean field model governed by a partial differential equation, representing a continuum of neural dynamics when the number of neurons tends to infinity, for which we give sufficient conditions for the system to possess an \textit{inertial manifold}, that is, a finite-dimensional manifold that attracts the population activity asymptotically in time.

In neuroscience, technology now makes it possible to record simultaneously the individual spikes of hundreds or thousands of neurons from the same brain region. Advanced statistical methods~\cite{perich2025neural} are then used to empirically estimate manifolds in latent spaces to which these signals are confined. This was performed across multiple species, including invertebrates, fish, and mammals, humans included. This led neuroscientists~\cite{sadtler2014neural,gallego2017neural} to introduce the concept of \textit{neural manifolds}, hypothesizing that the structure of coordinated population neural activity is constrained to manifolds with dimensions much smaller than the number of neurons in the population. This is consistent with theories stating that information is encoded at the population scale in the brain~\cite{averbeck2006neural}. Indeed, many brain regions contain thousands of neurons while processing or representing variables that are intrinsically low-dimensional, such as the orientation of a body part or motor commands~\cite{gallego2020long, chaudhuri2019intrinsic}. If these neuronal populations encode finite-dimensional inputs and/or outputs, then their large size necessarily implies a high degree of redundancy. Consequently, the relevant information is naturally expected to be represented collectively at the population level, constraining global activity patterns to a low-dimensional manifold whose dimension is comparable to that of the encoded variables.

In this paper, we derive and study the following model
\begin{equation}
    \label{sys:AbstractModel}
    \left\{\begin{array}{@{}l@{\;}l@{}}
        \partial_t u_1 + (\tau_1 + \nu_1 (-\Delta)^{n_1}) u_1 =& \sum_{m=1}^d \W_{1m} \sigma\left(\alpha_m u_m + \beta_m\right) + \W^\In_{1} \I,
        \\
        \hspace{3em} \vdots \hspace{3em} & \hspace{3em} \vdots
        \\
        \partial_t u_\ell + (\tau_\ell + \nu_\ell (-\Delta)^{n_\ell}) u_\ell =& \sum_{m=1}^d \W_{\ell m} \sigma\left(\alpha_m u_m + \beta_m\right) + \W^\In_{\ell} \I,
        \\
        \hspace{3em} \vdots \hspace{3em} & \hspace{3em} \vdots
        \\
        \partial_t u_d + (\tau_d + \nu_d (-\Delta)^{n_d}) u_d =& \sum_{m=1}^d \W_{dm} \sigma\left(\alpha_m u_m + \beta_m\right) + \W^\In_{d} \I
    \end{array} 
    \right. \hspace{2em} \text{ in } (0,+\infty) \times \Omega.
\end{equation}

Here, $u_\ell (t, x) \in \dR$ represents the neural activity at time $t \in [0, +\infty)$ of an idealized neuron located at position $x \in \Omega \subset \mathbb{R}^2$ within layer $\ell \in \{1, \dots, d\}$. The collection of fields $U = (u_1, \dots, u_d)$, defined over the domain $\Omega$, represents the activity of a neural sheet modeled as a continuous surface, with a biologically inspired organization consisting of an infinite population of neurons arranged into finitely many layers. The time evolution $\partial_t u_\ell$ in layer $\ell$ is governed by three mechanisms. Firstly, a dissipative layer-regularizing effect $(\tau_\ell+\nu_\ell (-\Delta)^{n_\ell} )$, decomposed into the leaking rate $\tau_\ell > 0$ and a $2 n_\ell$-order diffusive operator with strength $\nu_\ell \geq 0$. The second mechanism is the interaction with the rest of the population, corresponding to the term $\sum_{m=1}^d \W_{\ell m} \sigma\left(\alpha_m u_m + \beta_m\right)$. This term is the sum over all layers of the activity output of each field $\sigma(\alpha_m u_m + \beta_m)$, each formed with a sigmoid function $\sigma \colon \dR \to \dR$, an excitability gain function $\alpha_m$ and an excitability offset function $\beta_m$; these outputs are then combined through general linear operators $\W_{\ell m}$ acting on the underlying function spaces, which encode the connections between neurons (see Figure~\ref{fig:AnalogyDiscreteContinuum}). The last mechanism, in the form of the term $\W^{\In}_\ell \I$, represents a general input field $\I$ connected to the layer by a linear operator $\W^{\In}_\ell$. The input $\I$ takes values in a Banach space $V$, naturally chosen as a function space representing, for example, the field activity of another neural sheet or the external sensory environment.

In Section~\ref{subsec:FunctionFramework}, we introduce vector notations, so that the previous system~\eqref{sys:AbstractModel} is written in the compact form
\begin{equation}
    \label{eq:AbstractModel}
    \partial_t \mathrm{U} + (\vtau +\vnu (-\Delta)^{\vn})\mathrm{U} = \W \sigma(\upalpha \mathrm{U} + \upbeta) + \W^{\In}\mathcal{I} \hspace{2em} \text{ in } (0,+\infty) \times \Omega.
\end{equation}

This model, which we refer to as \textit{Layered Dissipative Neural Fields} (\textbf{LDNF}), is the main topic of study of this paper.
We here consider it on the two-dimensional periodic torus, $\Omega = \Torus^2 := \dR^2/\dZ^2$, leaving the treatment of boundary conditions for future work.

Our model is inspired by \textit{Continuous-Time Recurrent Neural Networks} (CTRNN)~\cite{barak2017recurrent, beer2006parameter,dayan2005theoretical}, also referred to as \textit{network rate models}. These models are written as Ordinary Differential Equations for the activities of a finite population of neurons. The linear operators $\W_{\ell m}$ are a natural generalization of \textit{connection weight matrices}, when replacing a finite population of neurons with a continuous neural medium defined on a surface, since in both the discrete and the continuum setting they act as linear operators on the state space. The Amari models~\cite{amari1975homogeneous, amari1977dynamics} are obtained by taking $W_{\ell m}$ to be integral operators and $\nu_\ell = 0$. We here consider a generalized structure in order to obtain abstract results encompassing a broad class of equations, by imposing that the connection operators $W_{\ell m}$ are bounded linear operators in Sobolev spaces of distributions, provided the order of each operator is at most half the order of the dissipative operator. This is motivated by cortical computations that appear to involve differential rather than integral information; orientation selectivity in the mammalian primary visual cortex is the standard example, see Section~\ref{sec:Derivation}. The dissipative operator $\vnu (-\Delta)^\vn$, allowing for more general connection operators, is derived in the activity-field limit from the local \textit{gap junction} regulation
mechanism, an ohmic spike synchronization of the finite neuron population model.

In Theorem~\ref{thm:InertialManifold}, we establish the existence of an inertial manifold for the equation~\eqref{eq:AbstractModel}, under conditions on the parameters $\vnu, \vn, \vtau$ and a regularity assumption on the operator $\W$. The \textit{inertial manifold} is a finite-dimensional Lipschitz manifold embedded in the infinite-dimensional function space in which the equation is posed. It exponentially attracts all trajectories and is positively invariant under the dynamics, thereby characterizing the asymptotic behavior of the solutions. The dissipative structure of the equation provides a crucial spectral gap condition ensuring the existence of inertial manifolds. For a coupling assumed only to be bounded between Sobolev spaces, the higher-order dissipation
$\vn \geq 2$ required by Theorem~\ref{thm:InertialManifold} is, to the best of our knowledge, the minimal mechanism currently known to guarantee the existence of inertial manifolds. The interest of the present class is that the whole
requirement collapses into a single integer, interpreted as the range of the local activity
regulation.

This model is interpreted as a mapping assigning to each input stimulus a behavioral response of the neural sheet. We construct the analogy between the concept of \textit{neural manifolds} and  \textit{inertial manifolds} as follows: on the one hand, we have a large number of neurons with a coordinated activity leading to a population activity lying on a low-dimensional manifold. On the other hand, we have an infinite continuum of neurons with dynamics imposed by the equation~\eqref{eq:AbstractModel} leading to a behavior on a finite-dimensional manifold. Figure~\ref{fig:InertialManifoldNeuralManifoldAnalaogy} illustrates this analogy. Our model provides a mechanism at the macroscopic scale for the empirically observed finite-dimensional state space of neural population activity patterns. The main result of this paper is the following upper bound on the dimension of the inertial manifold:
\begin{equation*}
         \dim \mathcal{M} \leq (\kappa_0 d)^{1+\delta}
         \left(2^{\tfrac{n_{\max}}{2}}\frac{\|\W\|_{\dL^2\to\dH^{-\vs}}\|\sigma'\|_{\infty}\|\upalpha\|_{\dL^\infty}\bigvee \tau_{\max}\bigvee 1}{\nu_{\min}\wedge \tau_{\min} \wedge 1}\right)^\delta + \Rank(\W^{\In}),
\end{equation*}
where $\delta := 2\left(n_{\min}-2\max_{\ell = 1}^d\left(s_\ell/n_\ell\right)\right)^{-1}$ and $\kappa_0 \geq 1$ is an absolute constant independent of the parameters. The dimension of the manifold is thus controlled by the strength of the inter-layer coupling relative to the intra-layer activity regulation, the exponent $\delta$ degrading as the order of the coupling approaches half the order of the dissipation. We require the input operator $\W^\In$ to be of finite rank because $\dim \mathcal{M}$ is finite while, near a stable stationary solution, the Banach implicit function theorem already gives
\begin{equation*}
    \Rank(\W^{\In}) \leq  \dim \mathcal{M}.
\end{equation*}
Since the pioneering works of Amari~\cite{amari1975homogeneous, amari1977dynamics} and Wilson \& Cowan~\cite{wilson1972excitatory}, there has been a continuing development of Neural Field theory. This approach has been used to study pattern formation \cite{dumont2024pattern}, controlled traveling waves \cite{richardson2005control}, and models of visual cortical processing \cite{bolelli2025neural, tamekue2024reproducibility}; see \cite{burlakov2016two, pinto2001spatially, coombes2014neural, cook2022neural} for broader perspectives. More recently, stochastic extensions have been introduced to account for noise effects \cite{kuehn2018gradient, carrillo2025well}.
As in the equation considered in this paper, additional diffusion terms in neural field equations have been derived and studied in several works~\cite{spek2020neural, avitabile2024well, spek2024hopf, salasnich2019power}.

This paper is structured as follows. In Section~\ref{sec:Derivation}, we derive the model as a formal mean field limit as the number of neurons goes to infinity in a CTRNN neural lattice model, with an additional local activity regulation representing the \textit{gap junction} mechanism of \textit{electrical synapses}. Nearest-neighbor coupling yields the diffusive term
of~\cite{spek2020neural,elvin2008pattern}. Since gap junctions act at the scale of individual
spikes whereas $u_\ell$ describes population activity, and since more extensively coupled
regions are naturally modeled as more strongly coordinated, we let the regulation act over
an $n_\ell$-neighborhood. This produces the $n_\ell$-th power of the Laplacian and hence a
larger class of equations, the classical diffusive case being recovered for $n_\ell = 1$. In Section~\ref{sec:MainRestult}, we present the results obtained in this paper. Specifically, Theorem~\ref{thm:WPSemigroup} is devoted to the existence and uniqueness of the solutions of the equation~\eqref{eq:AbstractModel} together with the existence of the associated parametric semigroup. In Theorem~\ref{thm:InertialManifold}, we state the main result on the existence of an inertial manifold and the dimensional upper bound. Subsequently, in Theorem~\ref{thm:UniversalApprox} and Theorem~\ref{thm:DynUniversalApprox}, we provide \textit{universal approximation} properties for our model, proving the expressive power of this class of equations. The first is a universal approximation theorem for exponentially attractive stationary patterns parametrized by the input, interpreted as the \textit{computational} expressivity of the model; the second is concerned with the expressivity of the equation over finite time intervals, approximating general nonlinear dynamics of field activity. Theorems~\ref{thm:BifStatDiag},~\ref{thm:BifStatDiagSelfAdjPerturbation} and \ref{thm:BifHopfDiag} provide sufficient conditions on the connection operator $\W$ for ensuring the existence of nontrivial stationary solutions and traveling waves. These results are then applied in the Propositions~\ref{prop:applicationSAsym},~\ref{prop:applicationSAPert}, and~\ref{prop:applicationTraveling} to toy model equations. Section~\ref{sec:GeneralizedDNF} discusses more general dissipative operators leading to the same results, thus relaxing the specific Laplacian structure. Finally, Section~\ref{sec:Proofs} is devoted to proving the stated results. In the next section, we introduce the notations and the functional framework used throughout the paper.

\begin{figure}[p]
    \centering
    \includegraphics[width=0.9\linewidth]{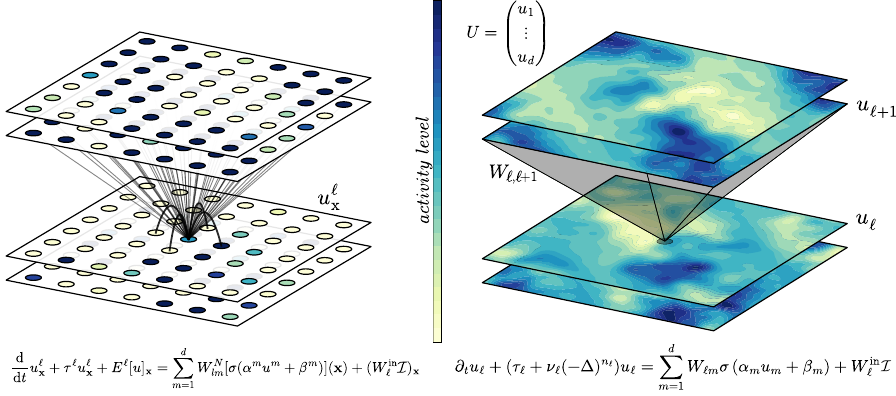}
    \caption{Representation of the connection operators between layers in both the discrete lattice model on the left and our continuous field model on the right.}
    \label{fig:AnalogyDiscreteContinuum}

    \vspace{2\baselineskip}

    \includegraphics[width=\linewidth]{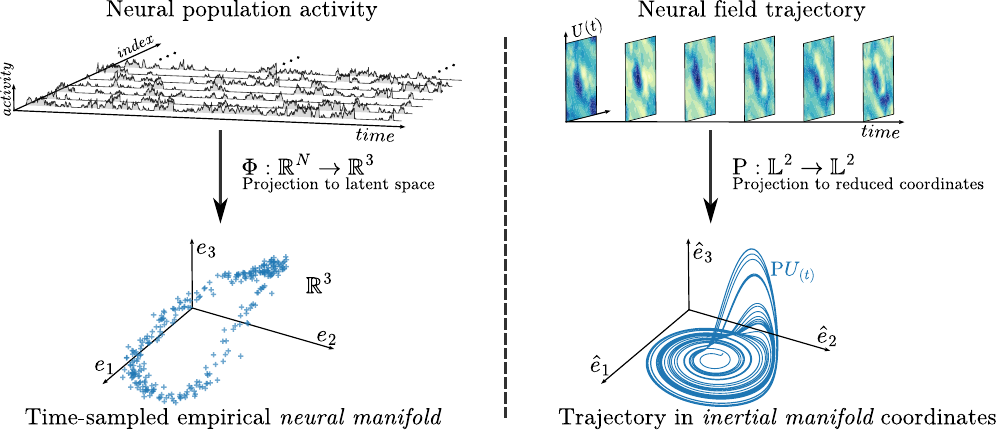}
    \caption{Illustrations of an empirical neural manifold and a neural field inertial manifold. On the left, inspired by \cite{mitchell2023neural}, we represent the projection of a neural population activity onto the neural manifold. On the right, we represent the projection of a solution to the neural field equation onto the inertial manifold coordinates, that is, a finite-dimensional subspace of $\dL^2$.}
    \label{fig:InertialManifoldNeuralManifoldAnalaogy}
\end{figure}

% \begin{figure}
%     \centering
%     \includegraphics[width=0.9\linewidth]{Figures/Figure1V1BW.pdf}
%     \caption{Representation of the connection operators between layers in both the discrete lattice model on the left and our continuous field model on the right.}
%     \label{fig:AnalogyDiscreteContinuum}
% \end{figure}

% \begin{figure}
%     \centering
%     \includegraphics[width=\linewidth]{Figures/InertialManifoldNeuralManifoldAnalogy.pdf}
%     \caption{Illustrations of an empirical neural manifold and a neural field inertial manifold. On the left, inspired by \cite{mitchell2023neural}, we represent the projection of a neural population activity onto the neural manifold. On the right, we represent the projection of a solution to the neural field equation onto the inertial manifold coordinates, that is, a finite-dimensional subspace of $\dL^2$.}
%     \label{fig:InertialManifoldNeuralManifoldAnalaogy}
% \end{figure}

%% file: preliminaries.tex
In this paper, we consider the equations on the two-dimensional torus $\Omega = \Torus^2 = \dR^2/\dZ^2$ representing a neural sheet. This mathematical assumption simplifies the analysis; nonetheless, the proofs could be adapted to a general smooth domain $\Omega \subset \dR^2$ with appropriate boundary conditions. 

For $s \in \dR$, we define the Sobolev spaces $H^s(\Omega, \dR)$ of distributions $u$ satisfying
\begin{equation*}
    \|u\|_{H^s}^2 =  \sum_{\xi \in \dZ^2} (1+4\pi^2|\xi|^2)^s |\hat{u}(\xi)|^2 < +\infty
\end{equation*}
where $\hat{u}(\xi)$ is the Fourier mode $\xi\in \dZ^2$ of $u$, defined as $\hat{u}(\xi) = \int u(x) e^{-i2\pi \xi \cdot x}\dd x$. In the following, we fix some integer $d \geq 1$, representing the number of layers of the neural sheet. For $\vs= (s_1, \cdots, s_d) \in \dR^d$, we define the product space $\dH^{\vs} = H^{s_1} \times \cdots \times  H^{s_d}$, endowed with the norm $ \|U\|_{\dH^{\vs}}^2 = \sum_{i = 1}^d \|u_i\|_{H^{s_i}}^2$, where $U = (u_1, \cdots, u_d) \in \dH^{\vs}$. We identify $\dL^2 = \dH^{\mathbf{0}} \cong L^2(\Omega, \dR^d)$ from the Parseval identity. For $\vr, \vs \in \dR^d$, we use the notation $\vr \leq \vs$ for the partial ordering on $\dR^d$, defined component-wise as: $\vr \leq \vs$ if and only if $r_\ell \leq s_\ell$ for all $\ell = 1, \cdots, d$, where $r_\ell$ and $s_\ell$ are the $\ell$-th components of $\vr$ and $\vs$.

We recall the equation~\eqref{eq:AbstractModel} studied in this paper,
\begin{equation*}
    \partial_t U + (\vtau + \vnu (-\Delta)^{\vn})U = \W \upsigma(\upalpha U + \upbeta) + \W^\In \I \; \text{ in } (0,\infty) \times \Torus^2,
\end{equation*}
and the spaces in which each term is defined For $\vn = ( n_1, \cdots, n_d ) \in \dN^d$ and $\vnu = (\nu_1, \cdots, \nu_d) \in \dR^d$, the operator $\vnu(-\Delta)^{\vn} \colon \dH^{2\vn} \to \dL^2$ is a sectorial operator with compact resolvent, whenever $\vnu > 0, \vn \geq 1$, and is defined layer-wise as $\vnu(-\Delta)^{\vn}U = (\nu_1(-\Delta)^{n_1}u_1, \cdots, \nu_d(-\Delta)^{n_d}u_d)^{\Trps}$. For $\vtau = (\uptau_1, \cdots, \uptau_d)^\Trps \in \dR^d$, the notation $\vtau U$ is defined as $\diag (\vtau) U = (\uptau_1 u_1, \cdots, \uptau_d u_d)^{\Trps}$. We suppose that $\upbeta = (\beta_1, \cdots, \beta_d)^\Trps \in \dL^2$, $\upalpha = (\alpha_1, \cdots, \alpha_d)^\Trps \in \dL^\infty$ (that is, $\alpha_\ell \in L^\infty$ for $\ell = 1, \cdots, d$), and we define $\upalpha U + \upbeta := (\alpha_1 u_1 + \beta_1, \cdots, \alpha_d u_d+\beta_d)^\Trps$. For some bounded Lipschitz function $\sigma \colon \dR \to \dR$, $\upsigma \colon \dR^d \to \dR^d$ is defined layer-wise as $\upsigma(U) = (\sigma(u_1), \cdots, \sigma(u_d))^\Trps$. Furthermore, $\W \colon \dL^2 \to \dH^{-\vs}$ is a bounded linear operator that can be decomposed as $\W = (\W_{\ell m})_{1\leq \ell,m\leq d}$, with $\W_{\ell m} \colon L^2 \to H^{-s_\ell}$. Thus explicitly, $\W U = (\sum_{m = 1}^d \W_{1m}u_m,\cdots,\allowbreak  \sum_{m = 1}^d \W_{dm}u_m )^\Trps$. We further suppose that $(V, \|\cdot\|_V)$ is a Banach space representing the space of inputs, that $\I \in V$, and that $\W^\In \colon V \to \dH^{-\vs}$ is a bounded linear operator. Finally, whenever we have a parameter $\boldsymbol{\gamma} = (\gamma_1, \cdots, \gamma_d)^\Trps \in \dR^d$, we use the notation $\gamma_{\min} = \min_{\ell = 1}^d \gamma_\ell, \gamma_{\max} = \max_{\ell = 1}^d \gamma_\ell$.

%% file: Derivation.tex
\newcommand{\bfx}{\mathbf{x}}
In this section, we derive the equation~\eqref{eq:AbstractModel} as a formal limit of a sequence of lattice neural models as the number of neurons grows to infinity. For $N \geq 2$, let $h = N^{-1}$ and denote by $\Omega_N$ the following lattice of $\Omega = \Torus^2$, 
\begin{equation*}
    \Omega_N := \{\bfx = (i h, j h), 1\leq i, j \leq N\}\subset \Omega.
\end{equation*}
We suppose that we have $N\times N\times d$ neurons in $d$ layers over the lattice $\Omega_N$ (see Figure~\ref{fig:AnalogyDiscreteContinuum}). For $1\leq \ell \leq d$, $t\geq 0$, we denote by $u^{\ell}_{\bfx}(t) \in \dR$ the activity of the neuron at position $\bfx \in \Omega_N$ in layer $\ell$ at time $t$. In the following, we omit the time dependence, and we denote by $u^\ell \in \dR^{\Omega_N} \cong \dR^{N^2}$ the activity in layer $\ell$ and $u \in \dR^{d\times \Omega_N} \cong \dR^{dN^2}$ the global activity. We then consider, with these notations, the associated continuous-time recurrent neural network (\textbf{CTRNN}), that is, the ordinary differential equation characterizing the evolution in time of the activity of each neuron
\begin{equation}
    \label{eq:LatticeModel}
    \left\{
    \begin{array}{@{}l@{\;}l@{}}
        \frac{\dd}{\dd t} u^{1}_{\bfx} + \tau_1 u^{1}_{\bfx} + E^1[u^1]_{\bfx} =& \sum_{m = 1}^d \W^N_{1m} [\sigma(\alpha^m u^m + \beta^m)](\mathbf{x}) + (\W^{\In,N}_1 \mathcal{I})_\bfx,
        \\
        \hspace{8em} \vdots  &  \hspace{8em}  \vdots
        \\
        \frac{\dd}{\dd t} u^{\ell}_{\bfx} + \tau_{\ell} u^{\ell}_{\bfx} + E^\ell[u^\ell]_{\bfx} =& \sum_{m = 1}^d \W^N_{\ell m} [\sigma(\alpha^m u^m + \beta^m)](\mathbf{x}) +( \W^{\In,N}_\ell \mathcal{I})_\bfx,
        \\
        \hspace{8em} \vdots  &  \hspace{8em}  \vdots 
        \\
        \frac{\dd}{\dd t} u^{d}_{\bfx} + \tau_d u^{d}_{\bfx} + E^d[u^d]_{\bfx} =& \sum_{m = 1}^d \W^N_{dm} [\sigma(\alpha^m u^m + \beta^m)](\mathbf{x}) + (\W^{\In,N}_d \mathcal{I})_\bfx,
    \end{array}\right. \text{ for }\bfx \in \Omega_N, t \geq 0.
\end{equation}

where $\tau_{\ell} >0$ is the leaking rate in layer $\ell$, $\mathcal{I} \in V$ is the external input, $\alpha^m , \beta^m  \in \dR^{\Omega_N}$ are respectively the excitability gain and the excitability offset in layer $m$, $\sigma \colon \dR \to \dR$ is a sigmoid function and we use the notation $\sigma(\alpha^m u^m + \beta^m)_{\bfx} := \sigma(\alpha^m_\bfx u^m_\bfx + \beta^m_\bfx)$. The linear \textit{connection operators} $\W^N_{\ell m} \colon \dR^{\Omega_N} \to \dR^{\Omega_N}$ represent the connection between the \textit{chemical} output $\sigma(\alpha^m u^m + \beta^m)$ of layer $m$ and the neurons in layer $\ell$; each of them is identifiable with a weight matrix in $\mathcal{M}_{N^2,N^2}(\dR)$, as in Artificial Neural Networks. The linear operator $\W^{\In,N}_{\ell} \colon V \to \dR^{\Omega_N}$ represents the connection of the input stimuli with the network. The function $E^\ell\colon \dR^{\Omega_N}\to\dR^{\Omega_N}$ models the intra-layer
\emph{activity regulation} mechanism arising from \emph{gap junction} coupling, that
is, from electrical synapses.

Electrical synapses are gap-junction-mediated connections providing reciprocal
pathways for ionic current between neurons. Two of their properties
shape the structure of $E^\ell$. They are \emph{homotypic}, coupling being found
predominantly between neurons of the same class, which justifies modeling the mechanism as acting
within a layer; and they are \emph{local}, formed between neighboring neurons. This mechanism varies substantially across cortical areas,
layers and cell classes \cite{fukuda2006gap}. The range
over which the mechanism acts therefore enters the model as a \emph{parameter} rather
than as a universal constant.

Functionally, electrical synapses coordinate spiking~\cite{bennett2004electrical}. This coordination depends on the shape of the action
potential, the local firing frequency and the phase. This complex mechanism acts as \textit{virtual chemical synapses} \cite{montbrio2020exact}. We do not
attempt to resolve this dependence here. What is relevant to a model posed at the
scale of an activity field is only its macroscopic imprint: a mechanism that smooths
the population activity over a neighborhood, with a power that increases as the
neighborhood widens.

We use the following biologically motivated hypotheses to model this regulation
mechanism: it is \emph{local}, and we also suppose here that it is \emph{isotropic}
and \emph{homogeneous}, so that we are motivated to consider
\begin{equation*}
    E^\ell[u^\ell]_\bfx \;=\; c^\ell_N\,g_\ell\!\left(\sum_{|i|+|j|\leq n_\ell}
    R^\ell_{ij}\bigl(u^\ell_{\bfx}-u^\ell_{\bfx+(ih,jh)}\bigr)\right),
\end{equation*}
where $c^\ell_N\geq0$, $g_\ell\colon\dR\to\dR$, $R^\ell_{ij}\in\dR$, such that in layer
$\ell$ the coupling is assumed to be between neurons in layer $\ell$ at a lattice
distance of at most $n_\ell\geq1$.

In the case $n_\ell=1$, the above hypotheses are sufficient to characterize the
conductance weights $R^\ell_{ij}$ up to a constant and 
\begin{equation*}
    E^\ell[u^\ell]_{\bfx} = c^\ell_N\,g_\ell\!\left(4u^\ell_{\bfx}-\bigl(
    u^\ell_{\bfx+(h,0)}+u^\ell_{\bfx+(-h,0)}+u^\ell_{\bfx+(0,h)}
    +u^\ell_{\bfx+(0,-h)}\bigr)\right),
\end{equation*}
with $g_\ell$ increasing and $g_\ell(0)=0$. The weights correspond to the discrete
Laplacian $(-\Delta_N)\colon\dR^{\Omega_N}\to\dR^{\Omega_N}$, as obtained in
\cite{elvin2008pattern, spek2020neural} to model gap junctions through \textit{ohmic coupling}.
%\cite{elvin2008pattern, spek2024hopf}.

In the case $n_\ell\geq2$, locality, isotropy and homogeneity no longer characterize
the conductance weights, and a modeling choice has to be made. We make the most
parsimonious one available to us. Only two features of the mechanism need to be
retained at this scale: it acts over a neighborhood whose extent varies across
cortical areas, and widening that neighborhood strengthens the averaging of the
activity. The spike-level process responsible for this is neither resolved by, nor
required at, the scale of $u^\ell$. We also note that $u^\ell$ is an \emph{activity}
field rather than a membrane potential, so that the ohmic law which singles out the
Laplacian at the level of voltage does not transfer verbatim to this scale, and the
order of the operator is not constrained to be two. We therefore retain the operator
already established in the case $n_\ell=1$ and iterate it: the $n_\ell$-th
power of the discrete Laplacian is a local, isotropic and homogeneous
operator whose regularizing power is indexed by a single integer, which we match
to the coupling range. The term we consider for a general
$n_\ell$ is thus
\begin{equation*}
    E^\ell[u^\ell]_\bfx \;=\; c^\ell_N\,g_\ell\!\left(\bigl((-\Delta_N)^{n_\ell}
    u^\ell\bigr)_\bfx\right)
    \;=\; c^\ell_N\,g_\ell\!\left(\sum_{|i|+|j|\leq n_\ell}R^\ell_{ij}\,
    u^\ell_{\bfx+(ih,jh)}\right),
\end{equation*}
with $R^\ell_{ij}$ the stencil convolution of order $n_\ell$ of the discrete
Laplacian. The appropriate scaling to obtain a nontrivial limiting model is $c^\ell_N = N^{2n_\ell}$ so that, applying the activity regulation function to a smooth function $\varphi \in C^{\infty}(\Omega, \dR)$ at lattice evaluation points, we have the limit
\begin{equation*}
    E^\ell[\pi_N \varphi]_\bfx = N^{2n_\ell}g_\ell\left(((-\Delta_N)^{n_\ell} \pi_N \varphi)_\bfx\right) \xrightarrow[N \to \infty]{} \nu_\ell (-\Delta)^{n_\ell}\varphi(\bfx),
\end{equation*}
where $\pi_N \colon C(\Omega, \dR) \to \dR^{\Omega_N}$ is the lattice projection and $\nu_\ell := (g_\ell )'(0) \geq 0$, since $g_\ell$ is increasing.

The model~\eqref{sys:AbstractModel} is obtained by assuming that the solutions of the equation~\eqref{eq:LatticeModel} converge in the limit $N \to \infty$ to a smooth field $U \in C^{1,2\vn}(\dR_+\times \Omega, \dR^d)$, the activity regulation converges as computed above, the operator $W^N_{\ell m}$ converges to a linear operator
$W_{\ell m} \colon L^2 \to H^{-s_\ell}$, $W^{\mathrm{in},N}_\ell$ to an operator
$W^{\mathrm{in}}_\ell \colon V \to H^{-s_\ell}$, $\alpha^m \in \dR^{\Omega_N}$ to
$\alpha_m \in L^\infty(\Omega)$, and $\beta^m \in \dR^{\Omega_N}$ to $\beta_m \in L^2(\Omega)$.

We conclude this section with a remark on the class of equations captured by the
model. The degenerate case $\vnu = 0$ with $\W\colon\dL^2\to\dL^2$ a convolution
operator leads to the Amari equation, so that~\eqref{sys:AbstractModel} generalizes it in two directions:
through the regulation term derived above, and through the class of operators $\W$,
which the mean field limit does not constrain to be bounded convolutions. That
unbounded limits are relevant is suggested by orientation selectivity in the
mammalian primary visual cortex, where neurons respond selectively to the orientation
of a luminance gradient in the retinal input \cite{hubel1959receptive}: a neural
sheet may therefore encode differential information about the field it receives, and
an operator carrying out such a computation need not be bounded on $\dL^2$: a first-order
differential operator maps $\dL^2$ only into $\dH^{-1}$. Our framework accommodates
general orders, $\W\colon\dL^2\to\dH^{-\vs}$, as long as $\vs\leq\vn$, and thus a larger
class of scalings of the connection weights $\W^N$.
% Additionally, according to neuronal statistical studies~ \cite{changizi2001principles}, synapses density seems to be constant across species with different nervous system sizes, whereas neuron density per unit volume scales as $(\textit{Number of neurons})^{-1/3}$ and the synapse per neuron as $(\textit{Number of neurons})^{1/3}$, these observations do not seem to coincide with Amari's model assumptions, and motivates different scalings, and thus potentially nonconvolutional operators.
% {\color{pink} Applying Theorem~\ref{prop:applicationTraveling} allows to check an expected property of $n$. As cardiac neurons are known to have dense gap junction connections, and have the property of stable 1-wave of burst activity. In Propososition~\ref{prop:CaridacBurst}, we prove that increasing $n$ isolates the bifurcation point of the 1-burst away from the other wave number traveling waves, thus making it more stable.
% }

%% file: Mainresults.tex
In this section, we present the results obtained for the equation~\eqref{eq:AbstractModel}. Namely, we establish the existence and uniqueness of the solutions together with the existence of the associated parametric semigroup, the existence of the inertial manifold under appropriate conditions, universal approximation theorems, and the existence of nontrivial stationary solutions and traveling waves.

\subsection{Dynamical System Theory and Inertial Manifolds}

\begin{thm}
\label{thm:WPSemigroup}
Suppose that $\vn = (n_1, \cdots, n_d) \in \dN^d$ with  $\vn \geq 0$, $\vnu > 0$, $\vtau > 0$, $\vn \geq \vs$, $\W \in \mathcal{L}(\dL^2, \dH^{-\vs})$, $\W^{\In} \in \mathcal{L}(V, \dH^{-\vs})$, $\upalpha \in \dL^\infty$, $\upbeta \in \dL^2$, and that $\sigma$ is Lipschitz and bounded.

Then, for any input $\mathcal{I} \in V$, and any initial condition $U_0 \in \dL^2$, there exists a unique global weak solution of the equation~\eqref{eq:AbstractModel} satisfying
\begin{equation*}
    U \in C_t (\dL^2) \cap L^2_{t,loc}(\dH^{\vn}), \text{ with } \partial_t U\in L^2_{t,loc}(\dH^{-\vn}).
\end{equation*}

Furthermore, the equation defines a strongly continuous semigroup parametrized by the input $\mathcal{I}$, that we denote by $\Psi$; $\Psi \colon \dR_+ \times \dL^2 \times V \to \dL^2$, where $\Psi(t;U_0,\mathcal{I})$ is the solution $U(t)$ of the equation~\eqref{eq:AbstractModel} at time $t$, with initial condition $U_0$ and input $\mathcal{I}$.
\end{thm}

We now introduce the notion of inertial manifolds.
    
\begin{definition}
    Given a subset $B\subset V$, an inertial manifold $\mathcal{M}$ for the equation~\eqref{eq:AbstractModel} is a finite-dimensional Lipschitz submanifold of $\dL^2 \times \W^\In B$, which is positively invariant and exponentially attractive; that is, for all $(U,\I) \in \dL^2\times V$ such that $(U, \W^\In \mathcal{I}) \in \mathcal{M}$,
    \begin{equation*}
        (\Psi(t; U, \I), \W^\In \mathcal{I}) \in \mathcal{M} \text{ for all } t\geq 0
    \end{equation*}
    and there exist $\lambda>0$ and $C_B\colon \dR_+\to \dR_+$ such that
    \begin{equation*}
        \dist((\Psi\left(t; U, \I), \W^\In \mathcal{I}\right), \mathcal{M}) \leq C_B(\|U\|_{\dL^2})e^{-\lambda t}, \text{ for all } U \in \dL^2, \mathcal{I} \in B.
    \end{equation*}
\end{definition}

From the Banach implicit function theorem, around a stable stationary solution, the inertial manifold is a local Banach isomorphism of $\W^\In V$, in particular this implies that $\Rank(\W^{\In}) \leq \dim \mathcal{M}$. Thus as long as the equation~\eqref{eq:AbstractModel} admits a stable stationary solution, a necessary condition for the existence of an inertial manifold is that $\W^\In$ is a finite-rank operator. We work under this assumption in the next Theorem, which states the main result of this paper.

\begin{thm}
    \label{thm:InertialManifold} 
    We suppose that $\vnu >0$, $\vtau >0$, $\vn \geq 2$ and $\W \in \mathcal{L}(\dL^2, \dH^{-\vs})$ and $\W^{\In} \in \mathcal{L}(V, \dH^{-\vs})$ with $\Rank(\W^{\In}) < +\infty$ and $\vs \leq \vn$, with the additional assumption that $\vs < \vn$ if $n_{\min} = 2$.

    Then, for all $r_{\In}>0$, the equation~\eqref{eq:AbstractModel} admits an inertial manifold $\mathcal{M} \subset \dL^2 \times \W^\In B_V(0,r_\In)$, with the dimension upper bound
    \begin{equation}
         \dim \mathcal{M} \leq (\kappa_0 d)^{1+\delta}
         \left(2^{\tfrac{n_{\max}}{2}}\frac{\|\W\|_{\dL^2\to\dH^{-\vs}}\|\sigma'\|_{\infty}\|\upalpha\|_{\dL^\infty}\bigvee \tau_{\max}\bigvee 1}{\nu_{\min}\wedge \tau_{\min} \wedge 1}\right)^\delta + \Rank(\W^{\In}),
    \end{equation}
    where $\delta := 2\left(n_{\min}-2\max_{\ell = 1}^d\left(s_\ell/n_\ell\right)\right)^{-1}$ % {\color{pink} check $\delta := 2\left(n_{\min}-4\max_{\ell = 1}^d\left(\frac{s_\ell}{2n_\ell}\right)\right)^{-1}$}
    and $\kappa_0 \geq 1$ is an absolute constant independent of the parameters. 
\end{thm}

\begin{remark}
Theorem~\ref{thm:InertialManifold} requires the diffusion to be at least of order 4, with $\vn\geq 2$; this stems from Weyl's law on the two-dimensional torus, as the construction of the inertial manifold relies on a spectral gap condition, Lemma~\ref{lem:SpectralGap}. Indeed, the Laplacian operator on the two-dimensional torus does not yield arbitrarily large spectral gaps.
\end{remark}

\begin{remark}
We prove the existence of such an inertial manifold, in the specific form of a \textit{graph manifold}. That is, denoting by $\P_N \colon \dL^2 \to \dL^2$ the projection operator onto the eigenspace spanned by the first $N$ eigenfunctions of $(\vtau + \vnu (-\Delta)^\vn)$, we prove the existence of a Lipschitz function $\Phi \colon \P_N \dL^2 \times \W^\In B_{r_\In} \to (I-\P_N)\dL^2$, such that
\begin{equation*}
    \mathcal{M} = \{(y + \Phi(y, F), F) \text{ for } y\in B_{\P_N \dL^2}(0, \rho(\|\W^\In\|_{V\to \dH^{-\vs}}r_\In)), F \in \W^\In B_{r_\In}  \},
\end{equation*}
where $\rho(\|\W^\In\|_{V\to \dH^{-\vs}}r_\In)$ is defined in Lemma~\ref{lem:Absorption}. Furthermore, the dynamics of the equation~\eqref{eq:AbstractModel} on $\mathcal{M}$ can be written as a finite-dimensional ODE.
% \begin{equation*}
%     \frac{\dd}{\dd t} \mathbf{y} + A\mathbf{y} = \P_N R(\mathbf{y} + \Phi(\mathbf{y}, \W^\In \mathcal{I}), \W^\In \mathcal{I}) \text{ in } (0,+\infty) \times B_{\P_N \dL^2}(0, \rho(\|\W^\In\|_{V\to \dH^{-\vs}}r_\In)).
% \end{equation*}
\end{remark}
    
\subsection{Universal Approximation Theorems}

   %\subsection{Static Approximation Theorem}
    Since the structure of our model is an infinite-dimensional extension of a \textbf{CTRNN}, and thus a dynamical version of Artificial Neural Networks (\textbf{ANN}), we can extend to Layered Dissipative Neural Field models both the Universal Approximation Theorem for \textbf{ANN}, see \textit{e.g.}~\cite{cybenko1989approximation}, and its dynamical counterpart for \textbf{CTRNN}, see \textit{e.g.}~\cite{funahashi1993approximation}. The following theorems require a compactness assumption in the input space. The natural choice of $V$ is a subspace representing field sensory input or the chemical output of another neural sheet. If we suppose that $V \subset L^2$, a sufficient condition for $K \subset \subset V$ to be compact is that $\|\I\|_{H^{\gamma}} \leq C$ for some $\gamma>0$ and for all $\I \in K$.

    The next theorem can be interpreted as a dynamical version of a universal approximation theorem. A computational task can be described by a function that associates an input activity with a desired pattern of cortical activity. In contrast to classical approximation results, where the output is obtained through a static input-output map, dynamical models realize this computation through the temporal evolution of their state. The theorem states that, given a continuous function mapping inputs into cortical activity and an arbitrary approximation error, there exists a dissipative neural field model whose input-driven dynamics attract the activity of one of its layers toward an approximation of the function output.
    
    \begin{thm}
        \label{thm:UniversalApprox}
        Let $\vn = (n_1, n_2)^\Trps \geq 1$, $\vnu= (\nu_1, \nu_2)^\Trps > 0$ and $\vtau = (\tau_1, \tau_2)^\Trps > 0$. For any uniformly continuous function $F\in C(V,L^2)$,  any compact subset $K\subset \subset V$ and any $\varepsilon >0$, there exist a constant $C(F, K, \vn, \vtau, \vnu, \varepsilon) >0$ and a two-layer dissipative neural field equation
        \begin{equation}
            \label{eq:UnivAppStaticLapprox}
            \begin{cases}
                \partial_t u_1 + (\tau_1 + \nu_1(-\Delta)^{n_1})u_1 &= W^\In_1 \mathcal{I},\\
                \partial_t u_2 + (\tau_2 + \nu_2(-\Delta)^{n_2})u_2 &= W_{21}\sigma(u_1 + \beta_1),
            \end{cases}
        \end{equation}
        such that any solution $(u_1(t; U^0, \mathcal{I}), u_2(t; U^0, \mathcal{I}))^\Trps$  of the system~\eqref{eq:UnivAppStaticLapprox} with initial condition $U^0 \in \dL^2$ and input $\mathcal{I} \in K$ satisfies
        \begin{equation*}
             \|u_2(t; U^0, \mathcal{I}) - F(\mathcal{I})\|_{L^2} \leq \max(\varepsilon , C(\|U^0\|_{\dL^2}\vee 1) e^{-\frac{\tau_{\min}}{4} t}).
        \end{equation*}
        In other words, on compact sets of inputs, the second layer approximates any continuous function of the input arbitrarily well, asymptotically in time.
    \end{thm}

    \begin{remark}
        In this case, the inertial manifold always exists, without the additional assumption that $\vn \geq 2$, and is given as a function of the input that can be chosen arbitrarily close to the target function on compact subsets. That is, there exists a Lipschitz bounded function $G^\varepsilon \colon V \to L^2\times L^2$, with $G^\varepsilon(\I) := (G^\varepsilon_1(\I), G^\varepsilon_2(\I))$ such that $\sup_{\I \in K} \|G_2^\varepsilon (\I) - F(\I) \|_{L^2} \leq \varepsilon$ and $\mathcal{M} := \{(G^\varepsilon(\I), (W^\In_1\I, 0)), \I \in V\}$ is the global attractor of the equation~\eqref{eq:UnivAppDyn2Lapprox}, in particular it is an inertial manifold for~\eqref{eq:UnivAppDyn2Lapprox}.
        
    \end{remark}
    
    This second theorem is concerned with dynamical aspects of neural field activity, that is, with how the system reacts to the input stimulus. We here prove that layered dissipative neural field equations approximate, arbitrarily well on finite time intervals, semi-linear parametrized field equations of the form~\eqref{eq:GeneralEquationUnivAppDyn}.
    
    \begin{thm}
        \label{thm:DynUniversalApprox}
        Suppose that $n \geq 1, \nu >0, \tau >0$ and $s<n$, and consider the well-posed scalar equation,
        \begin{equation}
            \label{eq:GeneralEquationUnivAppDyn}
            \partial_t u + (\tau + \nu (-\Delta)^n)u = R(u,\mathcal{I}) \text{ in } (0,+\infty) \times \Omega
        \end{equation}
        with $R : L^2 \times V \to H^{-s}$ Lipschitz continuous and bounded.
        
        Then, for any compact subset $K$ of $V$, for all $T>0, \varepsilon >0, \gamma > 0, \rho>0$, there exists a two-layer dissipative neural field equation
        \begin{equation}
            \label{eq:UnivAppDyn2Lapprox}
            \begin{cases}
                \partial_t u_1 + (\tau + \nu (-\Delta)^n)u_1 &= W_{11}\sigma(u_1 + \beta_1)  + W_{12}\sigma(\alpha_2 u_2) + W^{\In}_1 \mathcal{I},\\
                \partial_t u_2 + (\tau + \nu (-\Delta)^n)u_2 &= W_{21}\sigma(u_1 + \beta_1),
            \end{cases}
        \end{equation}
        
        with the following property. For all $u^0 \in B_{H^{\gamma}}(0,\rho)$ and $\I\in K$, there exists an initial condition $U^0 = (u_{1}^0, u_{2}^0)^\Trps$ such that the corresponding solutions satisfy
        \begin{equation*}
            \sup_{t \in [0,T]}\|u(t) - u_2(t)\|_{L^2} \leq \varepsilon,
        \end{equation*}
        where $u(t)$ denotes the solution of the equation~\eqref{eq:GeneralEquationUnivAppDyn} with initial condition $u^0$, and $U(t) := (u_1(t), u_2(t))^\Trps$ the solution of the system~\eqref{eq:UnivAppDyn2Lapprox} with initial condition $U^0$.
        
    \end{thm}
    
\subsection{Spontaneous Cortical Activity}
\label{subsec:SCA}
In this section, we establish nontrivial qualitative properties of the model in the absence of input stimulus, that is, for $\I = 0$. Namely, we derive sufficient conditions under which the model admits nonconstant stationary solutions and traveling waves, demonstrating its ability to reproduce essential dynamical features of cortical activity. These behaviors are particularly relevant as stationary patterns represent persistent states of neural activity, while traveling waves capture the propagation of activity across cortical regions. We consider one-layer equations over one-dimensional fields; such solutions can then be extended to be constant in the second spatial dimension to obtain solutions over two-dimensional fields. We consider the spaces $H^{s}(2\pi \Torus)$ for $s\in \dR$ and $L^2(2\pi \Torus) = H^0(2\pi \Torus)$, which we denote by $H^s$ and $L^2$. In the analysis below, we leverage the Fourier decomposition of the operators to study the kernel of the linearized operator at bifurcation points. We denote by $e_k := (x\mapsto e^{ikx})$, for $k \in \dZ$, the elements of the orthogonal basis of $L^2$ in Fourier modes, and by $H^s_{\dC} = H^s + i H^s$ the complex extension of the space $H^s$. For $n \geq 1$, we use the following standing assumptions on the linear operator $W$
\begin{assumption}
    \label{as:SCA-AssW}
    $W$ is a \emph{diagonal} operator of order $s \leq n$, that is, $W e_k = \hat{W}_k e_k$ for some $(\hat{W}_k)_{k\in \dZ}$ with $(\langle k\rangle^{-s}|\hat{W}_k|)_{k\in \dZ} \in \ell^{\infty}(\dC)$, equivalently $W \in \mathcal{L}(H^r, H^{r-s})$ for all $r \in \dR$.
\end{assumption}
We note that $ A^n_{\nu, \tau} := \nu(-\Delta)^n + \tau : H^{2n} \to L^2$ is a diagonal operator of order $2n$, with $\hat{A}^n_{\nu, \tau} = \langle k\rangle^n_{\nu, \tau} := \nu |k|^{2n} + \tau $. Let $F: H^{2n} \times \dR \to L^2$ be the function defined as
\begin{equation}
    \label{def:SCAFunctional}
    F(u,\lambda) = (\nu(-\Delta)^n + \tau)u  - \lambda \W \sigma (u),
\end{equation}
such that the equation~\eqref{eq:AbstractModel} reduces in this setting to
\begin{equation}
    \label{eq:dtFabstract1d}
    \partial_t u + F(u,\lambda) = 0 \text{ in } (0,\infty)\times \Torus.
\end{equation}
The parameter $\lambda \in \dR$ modulates the nonlinear coupling; mathematically, it serves as a bifurcation parameter. We first consider the problem of stationary solutions, that is, find $(u, \lambda) \in H^{2n}\times \dR$ such that
\begin{equation}
    \label{eq:StationnarySol1D1L}
    F(u,\lambda) = (\nu(-\Delta)^n + \tau)u  - \lambda \W \sigma (u) = 0.
\end{equation}

On the other hand, we consider the problem of traveling wave solutions of the equation~\eqref{eq:dtFabstract1d}, that is, to find $(u,\lambda) \in (C^1_{per, \kappa}(L^2)\cap C_{per, \kappa}(H^{2n})) \times \dR$, where $C_{per, \kappa}(X)$ is the space of $2\pi/\kappa$-periodic continuous functions from $\dR$ to a Banach space $X$.

Additionally, for $W$ satisfying~\eqref{as:SCA-AssW}, and given $A^n_{\nu, \tau}$, we define $\mathrm{c}_\ell$ as
\begin{equation}
    \label{def:SCAcell}
    \mathrm{c}_\ell:=  \frac{\hat{W}_\ell}{\langle \ell\rangle^n_{\nu,\tau}}.
\end{equation}
We shall also use the following two assumptions, on the operator $W$ and on the nonlinearity $\sigma$ respectively.
\begin{assumption}
     \label{as:SCA-AssW2}
     $W$ is such that $W e_0 = 0$.
\end{assumption}
Assumption~\eqref{as:SCA-AssW2} ensures that $u(t,x) \equiv 0$ is a solution of the equation~\eqref{eq:dtFabstract1d}.
\begin{assumption}
     \label{as:SCA-AssSigma}
    $\sigma$ is analytic in a neighborhood of $0$ and is normalized as $\sigma'(0) = 1$.
\end{assumption}
Through the parameter $\lambda \in \dR$, we can always rescale the first derivative of $\sigma$ to satisfy~\eqref{as:SCA-AssSigma}. We express its Taylor expansion around zero as $\sigma(x) = \sigma_0 + x + \frac{1}{2}\sigma_2 x^2 + \cdots + \frac{1}{k!}\sigma_k x^k + o(x^k)$.

This first theorem establishes the existence of bifurcating solution curves for the equation~\eqref{eq:StationnarySol1D1L} and also provides a result on their local dynamical stability.

\begin{thm}
    \label{thm:BifStatDiag}
    Let $n \geq 1$, $\nu > 0 $, $\tau >0$. Assume~\eqref{as:SCA-AssW}, \eqref{as:SCA-AssW2}, \eqref{as:SCA-AssSigma} and that $W$ commutes with the reflection operator. Let $k\in \dN$ be an \textbf{exceptional index}, that is, $\mathrm{c}_k$ is nonzero and $\mathrm{c}_k \neq \mathrm{c}_\ell$ for all $\ell \neq k$.
    Then, there exists a family of bifurcation curves $(s, \theta)\mapsto (u^k(\theta, s), \lambda^k(s)) \in H^{2n}\times \dR$, solutions of the stationary equation~\eqref{eq:StationnarySol1D1L}, such that  $\lambda^k(s) = \lambda^k_0 + \tfrac{s^2}{2} \lambda^k_2+ o(s^2)$, and
    \begin{equation*}
        u^k_{(\theta, s)}(x) =  \frac{s}{2\sqrt{\pi}} \left(e^{i(\theta+k x)}+e^{-i(\theta+k x)}\right) + o(s),
    \end{equation*}
    where
    \begin{equation*}
        \lambda^k_0 = \frac{1}{\mathrm{c}_k }, \; \lambda^k_2 = - \frac{1}{4\pi}\left(\lambda^k_0 \sigma_3+ (\lambda^k_0 \sigma_2)^2\frac{\hat{W}_{2k}}{\hat{A}_{2k} - \lambda^k_0 \hat{W}_{2k} }\right).
    \end{equation*}
    
    If $k$ is exceptional and furthermore $k = \argmax_{\ell \in \dN} \mathrm{c}_\ell$, then the nonlinear stability of the bifurcating solution is determined by $\lambda_2^k$, that is, if $\lambda_2^k > 0$ (resp. $\lambda_2^k < 0$) then for nonzero small $s$, the solution $u^k(\theta, s)$ is a locally asymptotically stable (resp. unstable) stationary solution of the equation~\eqref{eq:dtFabstract1d} modulo translations.
\end{thm}

\begin{prop}[Application of Theorem~\ref{thm:BifStatDiag}]
    \label{prop:applicationSAsym}
    The stationary equations
    \begin{align}
        \nu \partial^4_x u + \tau u = - \lambda \partial^2_x \sigma(u),\\[-1.5em]
        \intertext{and\vspace{-1.em}}
        -\nu \partial^2_x u + \tau u = \lambda a * \sigma(u)
    \end{align}
    admit nontrivial solutions, provided that $a \in H^{-s}$ is an even distribution and $\left(\int a(x) \cos (k x) \dd x /\langle k\rangle^n_{\nu,\tau}\right)_{k \in \dN}$ admits nonzero unique values.
\end{prop}

The following theorem is an extension of the previous result for self-adjoint perturbations of diagonal operators.

\begin{thm}
    \label{thm:BifStatDiagSelfAdjPerturbation}
    Let $n \geq 1$, $\nu > 0 $, $\tau >0$ and $\omega \mapsto W(\omega) \in C(U, \mathcal{L}(H^{s}, L^2))$ be an analytic mapping of self-adjoint operators, commuting with the reflection operator, such that $W(\omega) e_0 = 0$ for all $\omega \in U$. We expand $W$ as $W(\omega) = W_0 + \omega W_1 + o(\omega)$ and we suppose that $W_0$ satisfies~\eqref{as:SCA-AssW} and $\sigma$ satisfies~\eqref{as:SCA-AssSigma}.
    
    Then, for any exceptional index $k$ of the operator $W_0$, there exists a bifurcation surface, $(\omega, s) \mapsto (u^k(\omega, s), \lambda(\omega, s))$, with $\lambda(\omega, s) =  \lambda_{0} + \omega \lambda_{1} + o(\omega+s)$, and

    \begin{equation*}
        u^k(\omega, s)= s\phi_0+ s\omega\phi_{1} + o(s+s\omega), 
    \end{equation*}
    where $\lambda_0 = \hat{A}_k / \hat{W}_{0;k}$ and $\lambda_1 = -\lambda_0\langle \phi_0, W_1 \phi_0\rangle/\langle \phi_0, W_0 \phi_0\rangle$, and
    \begin{equation*}
        \phi_0 = \frac{1}{\sqrt{\pi}} \cos(k\,\cdot\,),\;\phi_1= (A^n_\tau-\lambda_0W_0)^{-1}(\lambda_0 W_1 \phi_0 + \lambda_1 W_0 \phi_0).
    \end{equation*}
\end{thm}

\begin{prop}[Application of Theorem~\ref{thm:BifStatDiagSelfAdjPerturbation}]
    \label{prop:applicationSAPert}
    The stationary equation
    \begin{equation}
        \nu \partial_x^4 u + \tau u = - \lambda \partial_x( b_\omega \partial_x \sigma(u)),
    \end{equation}
    with $b_\omega = 1+\omega b_1 + o(\omega)$, where $b_\omega$ is a smooth even function, admits nontrivial solutions, as does
    \begin{equation}
        -\nu \partial_x^2 u + \tau u = \lambda A_\omega[\sigma(u)],
    \end{equation}
    where $A_\omega[f](x) = \int a_\omega(x, y) f(y) \dd y$, for some smooth symmetric function $a_\omega(x, y) = a_\omega(y,x)$, analytic in $\omega$, satisfying an expansion of the form $a_\omega(x,y) = a_0(y-x) + \omega a_1 (x,y) + o(\omega)$.
\end{prop}

%\subsubsection{Traveling waves solutions}

Finally, the next theorem establishes an existence result for traveling wave solutions of the equation~\eqref{eq:dtFabstract1d}.

\begin{thm}
    \label{thm:BifHopfDiag}
    Let $n \geq 1$, $\nu > 0 $, $\tau >0$. Suppose that~\eqref{as:SCA-AssW}, \eqref{as:SCA-AssW2} and~\eqref{as:SCA-AssSigma} hold. Then, for any \textbf{Hopf exceptional index} $k\in \dN$, that is, $\Real(\mathrm{c}_k)$ is unique, nonzero and $\Imag(\hat{W}_k) \neq 0$, there exists a bifurcation curve $s \mapsto (\kappa(s), u(s), \lambda(s))$ solution of the equation~\eqref{eq:dtFabstract1d}, such that $u(s) \in C^{1}_{per, \kappa(s)}(Z)\cap C_{per, \kappa(s)}(X)$ with
    \begin{equation*}
        \kappa_0 = \Imag(\hat{W}_k)/\Real(\mathrm{c}_k) \text{ and }\kappa(s) = \kappa_0 + o(s) \text{ and }\lambda(s) = \frac{1}{\Real (\mathrm{c}_k)} + o(s).
    \end{equation*}
    Furthermore, $u$ satisfies the following expansion in $L^2$
    \begin{equation*}
        (u(s))(t, x) = \frac{s}{\sqrt{\pi}}\cos(k x-\kappa_0 t) + o(s).
    \end{equation*}
\end{thm}

\begin{prop}[Application of Theorem~\ref{thm:BifHopfDiag}]
    \label{prop:applicationTraveling}
    The equation
    \begin{equation}
        \partial_t u + \nu \partial^4_x u + \tau u = \lambda (-\partial^2_x + \gamma \partial_x ) \sigma(u),
    \end{equation}
    with $\gamma\neq 0$, admits traveling wave solutions, as does
    \begin{equation}
        \partial_t u -\nu \partial^2_x u + \tau u = \lambda a * \sigma(u),
    \end{equation}
    given that $\left(\int a(x) \cos (k x) \dd x /\langle k\rangle^n_{\nu,\tau}\right)_{k\in\dN}$ admits unique values and $\int a(x) \sin (k x) \dd x \neq 0$.
\end{prop}

% {\color{pink}
% \begin{prop}[Application to Neural Burst Waves Selection]
%     \label{prop:CaridacBurst}
%     Let $n \geq 1$ and suppose that $W$ a diagonal operator of order $s \leq n$, with the additional assumption that $s<0$ if $n =0$, let $\gamma >0$, and suppose that 
%     \begin{equation}
%         \Hat{W}_0 = 0, \Hat{W}_1 = 1 + \mathrm{i} \kappa_1 \text{ with }\kappa_1 \neq 0
%     \end{equation}
%     and consider the scalar dissipative neural field equation
%     \begin{equation*}
%         \partial_t u + \frac{\gamma}{2}(1+ (4\pi)^{-n}(-\Delta)^n)u= \lambda W\sigma(u).
%     \end{equation*}
%     Then, denoting $\lambda_{\min}$ the smallest positive Hopft bifurcation point different from $\lambda_1 = \gamma$ associated to the wave solution with wave number 1. We have the estimate,
%     \begin{equation}
%         \lambda_{\min} - \lambda_1 \geq \gamma\left(\frac{4^n}{(4\pi)^s\|W\|_{L^2 \to H^{-s}}} - 1\right)
%     \end{equation}
% \end{prop}}

%% file: GeneralizedLDNF.tex
In this section, we discuss a relaxed structure of the equation~\eqref{eq:AbstractModel} when allowing the linear operator $\vnu(-\Delta)^{\vn} : \dH^{2\vn} \to \dL^2$ to be generalized to a densely defined operator $A : D(A) \subset \dL^2 \to \dL^2$. Namely, we consider the equation
\begin{equation}
    \label{eq:AbstractModelGeneralized}
    \partial_t U +\vtau U + A U  = \W \sigma\left(\upalpha U + \upbeta\right) + \W^{\In}\mathcal{I} \hspace{2em} \text{ in } (0,+\infty) \times \Omega.
\end{equation}
This form corresponds to a generalized mean field limit of the \textit{gap junction} regulation mechanism $E^\ell[u]_{\mathbf{x}}$ of Section~\ref{sec:Derivation} under different biological assumptions, with $(E^\ell[u])_{1\leq \ell \leq d} \xrightarrow[N \to \infty]{} A U$. For example, preserving the \textit{local} and \textit{homotypic} intra-layer assumptions while removing the \textit{isotropic} and \textit{homogeneous} ones leads to operators of the form 
\begin{equation*}
    AU = \left((-1)^{n_1}\sum_{k,m = 0}^{n_1} \partial^{n_1-k}_1 \partial^{k}_2(a^1_{k,m}\partial^{n_1-m}_1 \partial^{m}_2 u_1), \cdots, (-1)^{n_d}\sum_{k,m = 0}^{n_d} \partial^{n_d-k}_1 \partial^{k}_2(a^d_{k,m}\partial^{n_d-m}_1 \partial^{m}_2 u_d)\right)
\end{equation*}
where $U = (u_1, \cdots, u_d)^\Trps$ and $a^\ell \colon \Omega \to \mathrm{Sym}_{n_\ell +1}(\dR) $ is a symmetric matrix field over $\Omega$.

Working with powers of the Laplacian and the associated Sobolev scale makes it possible to state the assumptions in terms of natural norms, and hence to obtain conditions expressed directly in the parameters. Nevertheless, every result above except those of Section~\ref{subsec:SCA} has an analog for the equation~\eqref{eq:AbstractModelGeneralized}, as we now explain. The assumptions are as follows. Let $\vtau > 0$ and $A : D(A) \subset \dL^2 \to \dL^2$ be a self-adjoint, densely defined sectorial operator with compact resolvent. If for some $0 \leq \alpha < 1$, we have $\|(\vtau + A)^{-\alpha} W\|_{\dL^2\to \dL^2} < +\infty$ and $\|(\vtau + A)^{-\alpha} W^\In\|_{V\to \dL^2} < +\infty$, then the equation~\eqref{eq:AbstractModelGeneralized} defines a strongly continuous parametric semigroup, see \textit{e.g.}~\cite[Thm~3.3.3]{henry2006geometric}. Denote by $\Sigma(\vtau + A) = \{0<\lambda_1\leq \cdots \leq \lambda_n\to \infty\}$ the spectrum of $(\vtau + A)$, counted with multiplicity. Suppose that, for some $0<\beta<\tfrac{1}{2}$ and for every arbitrarily large $M>0$, there exists $N \geq 1$ such that the following spectral gap holds
\begin{equation}
    \lambda_{N+1} > M
    \;\text{ and }\;
    \lambda_{N+1} - \lambda_{N} > M(\lambda_{N+1}^\beta + \lambda_{N}^\beta),
\end{equation}
that is, the analog of Lemma~\ref{lem:SpectralGap}. Then, if $\|(\vtau + A)^{-\beta} W\|_{\dL^2\to \dL^2} < +\infty$ and $\|(\vtau + A)^{-\beta} W^\In\|_{V\to \dL^2} < +\infty$, and $\Rank(W^\In) < +\infty$, the semigroup associated with the equation~\eqref{eq:AbstractModelGeneralized} admits a finite-dimensional inertial manifold for any bounded subset of $V$. In other words, this is the analog of Theorem~\ref{thm:InertialManifold} without the quantitative bound on the dimension of the inertial manifold. Finally, if $A$ is associated with an eigenbasis $\{e_n\}_{n\geq 0}$ of $\dL^2$, then the analogs of Theorem~\ref{thm:UniversalApprox} and Theorem~\ref{thm:DynUniversalApprox} hold.

%% file: ProofsAbstractResult.tex
\subsection{Well-posedness and Semigroup Theory}
We begin with a result on higher-order parabolic equations with singular source terms that serves as an \textit{a priori} estimate for the equation~\eqref{eq:AbstractModel}.

\begin{lem}
    \label{lem:LinearEqAbs}
    Suppose that $\vn = (n_1, \cdots, n_d) \in \dN^d$ with  $\vn \geq 0$, $\vnu > 0$, $\vtau > 0$, $\vn \geq \vs$. For any initial condition $U^0 \in \dL^2$ and any source term $G \in L^2_{t,loc}(\dH^{-\vs})$ there exists a unique weak solution to the equation
    \begin{equation}
        \label{eq:LinearEqAbs}
        \partial_t U + (\vtau + \vnu (-\Delta)^{\vn}) U = G
    \end{equation}
    satisfying
    \begin{equation*}
        U \in C_t (\dL^2) \cap L^2_{t,loc}(\dH^{\vn}), \text{ with } \partial_t U \in L^2_{t,loc}(\dH^{-\vn}),
    \end{equation*}
    such that $U(0) = U^0$. Furthermore, any solution satisfies in the distributional sense
    \begin{equation}
        \label{est:LinearEqAbs}
        \frac{\dd}{\dd t} \int |U|^2 \dd x + \nu_{\min}\int |\nabla^{\vn}U|^2 \dd x + \tau_{\min}\int |U|^2 \dd x \leq c_1(\vn, \vtau, \vnu) \|G\|^2_{\dH^{-\vs}},
    \end{equation}
    with $|\nabla^{\vn} U|^2 := \sum_{\ell = 1}^d |\nabla^{n_\ell} u_\ell|^2$, where $|\nabla^{k} u|$ denotes the Frobenius norm, and with $c_1(\vn, \vtau, \vnu) := 2^{(n_{\max}-1)} (\tau_{\min}\wedge \nu_{\min})^{-1}$, or in integral form,
    {\small \begin{equation}
        \label{est:LinearEqAbsInt}
        \left(\int|U(t)|^2 \dd x - \int|U^0|^2 \dd x\right) + \nu_{\min} \int^t_0\int  |\nabla^{\vn} U|^2\dd x +  \tau_{\min} \int^t_0\int |U|^2 \leq c_1(\vn, \vtau, \vnu) \int_0^t\|G(s)\|^2_{\dH^{-\vs}} \dd s.
    \end{equation}}
\end{lem}

\begin{proof}
    The existence follows from the classical Fourier--Galerkin approximation together with the \textit{a priori} estimate~\eqref{est:LinearEqAbs}. We first derive this estimate for smooth solutions. If $U$ is a regular solution of the equation~\eqref{eq:LinearEqAbs}, then it satisfies the energy estimate
    \begin{equation*}
        \frac{\dd}{\dd t} \int \frac{|U|^2}{2} \dd x + \int U \vnu (-\Delta)^{\vn} U \dd x + \int \left|\sqrt{\vtau} U\right|^2 \dd x = \int U \cdot G \dd x.
    \end{equation*}
    Integrating by parts in the second term and using the $\dH^{\vs}-\dH^{-\vs}$ duality in the last term, we obtain
    \begin{equation*}
        \frac{\dd}{\dd t} \int \frac{|U|^2}{2} \dd x + \nu_{\min} \int |\nabla^{\vn}U|^2 \dd x + \tau_{\min} \int |U|^2 \dd x \leq \|U\|_{\dH^{\vs}}\|G\|_{\dH^{-\vs}}.
    \end{equation*}

    For all $U \in \dH^{2\vn}$, with $\vs\leq \vn$, we can derive from the Parseval identity the estimate
    \begin{equation}
        \label{est:H2tau2nunablan2}
        \|U\|^2_{\dH^{\vs}} \leq 2^{(n_{\max}-1)}(\tau^{-1}_{\min}\vee \nu^{-1}_{\min}) \left(\nu_{\min} \int |\nabla^{\vn}U|^2 \dd x + \tau_{\min} \int |U|^2 \dd x \right).
    \end{equation}

    Using this in the previous estimate together with Young's product inequality and rearranging the terms, we obtain estimate~\eqref{est:LinearEqAbs}.
    
    We now apply this estimate to the Fourier--Galerkin approximations. Consider the approximate problems obtained by truncating the Fourier expansion of the source term $G$ as $G_N$, and let $U^N$ be the corresponding regular solutions. The sequence of solutions satisfies the previous estimate uniformly, so that the Banach--Alaoglu theorem yields a limit $U$, weak-$*$ in $L^\infty_{t,loc}(\dL^2)$ and weak in $L^2_{t,loc}(\dH^{\vn})$, satisfying estimate~\eqref{est:LinearEqAbsInt}.
    
    The uniqueness follows from the same energy estimate, as the difference of two solutions solves the homogeneous equation. By the Lions--Magenes lemma, the energy identity remains valid for this weak solution, and the previous estimate applied with $G\equiv0$ implies that the solutions coincide in $C_t(\dL^2)$.
\end{proof}

We consider in this section the equation
\begin{equation}
    \label{eq:AbstractModelFin}
    \partial_t \mathrm{U} + (\vtau +\vnu (-\Delta)^{\vn})\mathrm{U} = \W \sigma\left(\upalpha \mathrm{U} + \upbeta\right) + F_{\In} \hspace{2em} \text{ in } (0,+\infty) \times \Omega.
\end{equation}
with $F_{\In} \in \W^{\In} V \subset \dH^{-\vs}$.

\begin{lem}
    \label{lem:LipNBoundedR}
    We define the nonlinear operator $R : \dL^2 \times \W^\In V \to \dH^{-\vs}$ as
    \begin{equation}
        \label{def:Rdefinition}
        R(U, F_\In) = W\sigma(\upalpha U + \upbeta) + F_\In.
    \end{equation}
    Then $R$ satisfies the following estimates:
    \begin{equation}
        \label{est:LipschitzR}
        \|R(U_1, F_1)-R(U_2, F_2)\|_{\dH^{-\vs}} \leq L_1 \|U_1-U_2\|_{\dL^2} + \|F_1-F_2\|_{\dH^{-\vs}},
    \end{equation}
    for all $U_1, U_2 \in \dL^2$, $F_1, F_2 \in \W^\In V$ with $L_1 = \|W\|_{\dL^2\to\dH^{-\vs}}\|\sigma'\|_{\infty}\|\upalpha\|_{\dL^\infty}$ and
    \begin{equation}
        \|R(U,F_\In)\|_{\dH^{-\vs}} \leq \|W\|_{\dL^2\to \dH^{-\vs}}\|\sigma\|_{\infty} +\|F_\In\|_{\dH^{-\vs}}
    \end{equation}
    for all $U\in \dL^2, F_\In \in \W^\In V$.
\end{lem}

\begin{lem}
    \label{lem:FractionalResolventEstimateA}
    %We have the following estimate of the fractional resolvents of $(\vtau + \vnu (-\Delta)^\vn) : \dH^{\mathrm{m}+2} \to \dH^{\mathrm{m}+2}$ PHOTO 14 Aout
    Let $\vtau >0$, $\vnu >0$, $\vn \geq 0$ and $\alpha \geq 0$ be such that $\vs \leq 2\alpha \vn$. Then
    \begin{equation}
        \label{est:FractionalResolventEstimateA}
        \|(\vtau + \vnu (-\Delta)^\vn)^{-\alpha}\|_{\dH^{-\vs}\to \dL^2} \leq 2^{\alpha(n_{\max}-1)}(\tau_{\min} \wedge\nu_{\min} )^{-\alpha}.
    \end{equation}
\end{lem}

\begin{proof}
    Using the Fourier representation of the fractional power of the operator explicitly in the definition of the $\dH^{-\vs}$ norm, we obtain
    \begin{equation*}
        \|(\vtau + \vnu (-\Delta)^\vn)^{-\alpha} U\|_{\dL^2}^2 = \sum_{\ell = 1}^d \sum_{\xi \in \dZ^2} (\tau_\ell + \nu_\ell (4\pi^2|\xi|^2)^{n_\ell})^{-2\alpha}|\hat{u}_\ell(\xi)|^2.
    \end{equation*}
    Using the inequality
    \begin{equation*}
        (\tau_\ell + \nu_\ell X^{n_\ell})^{-2\alpha} \leq 2^{2\alpha(n_{\max}-1)}(\tau_{\min}\wedge\nu_{\min})^{-2\alpha}(1+X)^{-2\alpha n_\ell} \text{ for all } X\geq 0,
    \end{equation*}
    using that $s_\ell \leq 2\alpha n_\ell$, we obtain the bound
    \begin{equation*}
        \sup_{\xi\in \dZ^2} (\tau_\ell + \nu_\ell (4\pi^2|\xi|^2)^{n_\ell})^{-2\alpha} (1+4\pi^2|\xi|^2)^{s_\ell} \leq 2^{2\alpha(n_{\max}-1)}(\tau_{\min}\wedge\nu_{\min})^{-2\alpha}.
    \end{equation*}
    Using the previous estimate in the norm identity, we obtain
    \begin{align*}
        \|(\vtau + \vnu (-\Delta)^\vn)^{-\alpha} U\|_{\dL^2}^2 &= \sum_{\ell = 1}^d \sum_{\xi \in \dZ^2} (\tau_\ell + \nu_\ell (4\pi^2|\xi|^2)^{n_\ell})^{-2\alpha}|\hat{u}_\ell(\xi)|^2, \\
        &\leq 2^{2\alpha(n_{\max}-1)}(\tau_{\min}\wedge\nu_{\min})^{-2\alpha} \sum_{\ell = 1}^d \sum_{\xi \in \dZ^2} (1 + 4\pi^2|\xi|^2)^{-s_\ell}|\hat{u}_\ell(\xi)|^2,\\
        &= 2^{2\alpha(n_{\max}-1)}(\tau_{\min}\wedge\nu_{\min})^{-2\alpha} \|U\|_{\dH^{-\vs}}^2.
    \end{align*}
    Taking the square-root on both sides leads to the stated estimate.
    
\end{proof}

We now prove Theorem~\ref{thm:WPSemigroup}.

\begin{proof}[Proof of Theorem~\ref{thm:WPSemigroup}]
    The existence and uniqueness for short time follows from the Banach fixed point theorem in the Banach space $(X_T, \|\cdot\|_{X_T}) = (C([0,T], \dL^2), \sup_{t\in [0,T]}\|\cdot_t \|_{\dL^2})$, for some $T>0$ to be fixed. Define the mapping $\mathcal{F} \colon X_T \to X_T $ as $\mathcal{F}(U)$ is the unique solution of the equation~\eqref{eq:LinearEqAbs}, given by Lemma~\ref{lem:LinearEqAbs}, with source term $G = R(U,W^\In \I) \in C([0,T], \dH^{-\vs})$. We now show that the mapping is a strict contraction for $T>0$ sufficiently small. Take two elements $U^i \in X_T$, $i = 1,2$, and set $Z^i = \mathcal{F}(U^i)$. The difference $\overline{Z} = Z^1 -Z^2$ is a solution of the equation~\eqref{eq:LinearEqAbs} with source term $\overline{G} = R(U^1,F)- R(U^2, F)$. Then by Lemma~\ref{lem:LipNBoundedR}
    \begin{align*}
        \|\overline{G}\|_{\dH^{-\vs}} &\leq L_1 \|U^1- U^2\|_{\dL^2}.
    \end{align*}
    From the estimate~\eqref{est:LinearEqAbs} of Lemma~\ref{lem:LinearEqAbs}, for $\overline{Z}$, using the previous estimate and Gr\"onwall's lemma, we obtain
    \begin{equation*}
        \sup_{t\in [0,T]} \|Z^1(t)-Z^2(t)\|_{\dL^2} \leq C_1 T^{1/2} \sup_{t\in [0,T]} \|U^1(t)-U^2(t)\|_{\dL^2},
    \end{equation*}
    where $C_1 > 0$ is some constant independent of $U^1$ and $U^2$. Choosing $T =(2C_1)^{-2}$\ gives that $\mathcal{F}$ is a strict contraction on $X_T$. Thus the equation~\eqref{eq:AbstractModel} admits a unique solution on $[0,T]$. The global existence follows from a continuation argument, as the solution satisfies the no-blow-up estimate
    \begin{equation*}
        \sup_{t\in[0,T]} \|U(t)\|_{\dL^2} \leq \|U^0\|_{\dL^2} + C(T) ( \|W\|_{\dL^2\to \dH^{-\vs}}\|\sigma\|_{\infty} +\|W^\In\I\|_{\dH^{-\vs}}).
    \end{equation*}

    This ensures the existence of a parametric semigroup $\Psi\colon\dR_+ \times \dL^2 \times V \to \dL^2$; it remains to show that it is continuous. Take a sequence $(t^N, U^{0,N}, \I^N) \in \dR^+ \times \dL^2\times V$ converging to $(t, U^{0}, \I)$. For each $N\geq 1$, denote by $U^N$ the solution with input $\I^N$ and initial condition $U^{0,N}$, and by $U$ the solution with input $\I$ and initial condition $U^0$. Then the difference $\overline{U}^N = U-U^{N}$ is a solution of the equation~\eqref{eq:LinearEqAbs} with source term $\overline{G}^N = W \upsigma (\upalpha U + \upbeta) - W \upsigma (\upalpha U^N + \upbeta) + W^\In(\I - \I^N)$. Using the same Lipschitz estimation as for the existence result, from estimate~\eqref{est:LinearEqAbs} and Gr\"onwall's lemma, we obtain that
    \begin{equation*}
        \sup_{t\in [0,T]} \|U(t) -U^{N}(t)\|_{\dL^2} \leq C_{2} (\|W^\In\|_{V\to \dH^{-\vs}}\|\I-\I^N\|_{V} + \|U^0 -U^{0,N}\|_{\dL^2}),
    \end{equation*}
    for some constant $C_2>0$ independent of $\I,\I^N,U^0,U^{0,N}$. We conclude that the semigroup is continuous by noting that
    \begin{equation*}
        \begin{split}
            \|\Psi(t^N;U^{0,N}, \I^N) - \Psi(t;U^{0}, \I)\|_{\dL^2} &\leq  \|\Psi(t^N;U^{0}, \I) - \Psi(t;U^{0}, 
            \I)\|_{\dL^2}\\
            &\hspace{1em} + \|\Psi(t^N;U^{0,N}, \I^N) - \Psi(t^N;U^{0}, \I)\|_{\dL^2} 
        \end{split}
    \end{equation*}
    the first term converges to zero from the continuity in time of the solution and the second from the uniform continuity on bounded time intervals of the solution with respect to the input and the initial data, following estimate~\eqref{est:LinearEqAbs} on the difference of two solutions and Gr\"onwall's lemma. 
\end{proof}

\subsection{Existence of the Inertial Manifold}

\input{InertialManifold}

\subsection{Universal Approximation Theorems}
% In this section, we construct a family of exponentially contractive equations with stationary solutions that approximate arbitrarily well any one-layer artificial neural network; the compactness assumptions then allows us to apply the classical universal approximation theorem to obtain the required result.

\begin{proof}[Proof of Theorem~\ref{thm:UniversalApprox}]
    Let $\varepsilon>0$, as $K$ is compact in $V$, there exist $\rho_K > 0$ such that $K \subset B_V(0,\rho_K)$ and $N_0 \geq 1$ with linearly independent vectors $v_1, \cdots, v_{N_0} \in K$ such that  the finite-dimensional space $E_{N_0} = \Span\{v_1, \cdots, v_{N_0}\}$ satisfies $\dist(K,E_{N_0} \cap B_V(0,\rho_K)) \leq \frac{\varepsilon}{4}$. We consider the projection operator $P_{N_0} \colon V \to V$ onto $E_{N_0}$ and the coordinate isomorphism $\Pi_{N_0} \colon \dR^{N_0} \to E_{N_0}$ and its associated coordinate chart $\Pi^{-1}_{N_0} \colon  E_{N_0} \to \dR^{N_0}$, that we suppose to be normalized. Then, using the uniform continuity of $F$ and that $F(E_{N_0} \cap B_V(0,\rho_K))$ is compact in $L^2$, there exist $N_1 \geq 0$, $g_1, \cdots, g_{N_1} \in L^2$ and $f^{N_0,N_1} \colon \dR^{N_0} \to \dR^{N_1}$, that we denote as $f^{N_0,N_1}(x) =  (f^{N_0,N_1}_1(x), \cdots, f^{N_0,N_1}_{N_1}(x))^\Trps$, such that
    \begin{equation}
        \label{est:ApproximationSatic0to1}
        \sup_{\I \in K} \|F(\mathcal{I}) - F_{N_0,N_1}(\I)\|_{L^2} \leq \frac{\varepsilon}{2},
    \end{equation}
    where $F_{N_0,N_1}(\mathcal{I}) = \sum_{n = 1}^{N_1} g_n f^{N_0,N_1}_n(\Pi_{N_0}^{-1} P_{N_0}\I)$. Furthermore, we can assume that $\|g_n\|_{L^2} = 1$ for all $n = 1, \cdots, N_1$ and using the density of $H^{2n_2}$ in $L^2$, we assume that $g_1, \cdots, g_{N_1} \in H^{2n_2}$.

    Applying the Universal Approximation theorem, see \textit{e.g.}~\cite{cybenko1989approximation}, there exist $M\geq 1$, $A \in \mathcal{M}_{M,N_1}(\dR)$, $B \in \mathcal{M}_{N_0,M}(\dR)$, $\theta \in \dR^{M}$, such that
    \begin{equation*}
        \sup_{x\in B(0,\rho_K)} \|f^{N_0,N_1}(x)-f^{N_0,M,N_1}(x)\| \leq \frac{\varepsilon}{2N_1},
    \end{equation*}
    where $f^{N_0,M,N_1}_n(x) = \sum_{m=1}^M A_{m,n}\sigma ( \sum_{k = 1}^{N_0} B_{k,m} x_k+ \theta_m)$. This implies that
    \begin{equation}
        \label{est:ApproximationSatic1to2}
        \sup_{\mathcal{I}\in K} \|F_{N_0,N_1}(\I)-F_{N_0,M,N_1}(\I)\|_{L^2} \leq \frac{\varepsilon}{2},
    \end{equation}
    where $F_{N_0,M,N_1}(\I) =\sum_{n = 1}^{N_1} g_n f^{N_0,M,N_1}_n(\Pi_{N_0}^{-1} P_{N_0} \I)$.

    In order to construct the equation, we introduce the following functional lifting. Let $(\varphi_m)_{1\leq m \leq M} \in H^{2n_1}(\Omega)$ be a family of smooth bump functions associated with disjoint rectangular subsets $(\Omega_m)_{1\leq m \leq M}$ of $\Omega$, with $|\Omega_m| > 0 $ for all $m = 1, \cdots, M$, such that
    \begin{equation*}
        \varphi_m \equiv \frac{1}{|\Omega_m|} \delta_{m\ell} \text{ on } \Omega_\ell.
    \end{equation*}
    That is, $\varphi_m$ is equal to the inverse of the volume of $\Omega_m$ on $\Omega_m$ and zero on the other domains $\Omega_\ell$. We define $\beta_1 = \sum_{m=1}^M \varphi_m \theta_m \in L^2$ and the following auxiliary linear operators
    \begin{align*}
        \widetilde{W}^\In_1 = \sum_{m=1}^M \sum_{k=1}^{N_0} \varphi_m B_{km} (\Pi_{N_0}^{-1} P_{N_0} \cdot)_k \colon V \to H^{2n_1},\;
        \widetilde{W}_{21} = \sum_{n=1}^{N_1} \sum_{m=1}^{M} g_n A_{mn} \langle \Indic_{\Omega_m} , \cdot \rangle_{L^2} \colon L^2 \to H^{2n_2},
    \end{align*}
    such that by construction
    \begin{equation*}
        F_{N_0,M,N_1} = \widetilde{W}_{21}\sigma(\widetilde{W}^\In_1 \cdot + \beta_1).
    \end{equation*}
    Using the regularity of the auxiliary operators, we define the connection operators
    \begin{equation*}
        W^\In_1 = (\tau_1 + \nu_1(-\Delta)^{n_1})\widetilde{W}^\In_1 \colon V \to L^2
    \end{equation*}
    and
    \begin{equation*}
        W_{21} = (\tau_2 + \nu_2(-\Delta)^{n_2})\widetilde{W}_{21}  \colon L^2 \to L^2,
    \end{equation*}
    and consider the two-layer dissipative neural field equation

    \begin{equation}
        \label{eq:UnivAppStaticLapproxInproof}
        \begin{cases}
            \partial_t u_1 + (\tau_1 + \nu_1(-\Delta)^{n_1})u_1 &= W^\In_1 \mathcal{I},\\
            \partial_t u_2 + (\tau_2 + \nu_2(-\Delta)^{n_2})u_2 &= W_{21}\sigma(u_1 + \beta_1).
        \end{cases}
    \end{equation}

    We have that for any $\mathcal{I} \in V$, $(\widetilde{W}^{\In}_1 \mathcal{I}, F_{N_0,M,N_1}(\mathcal{I}))^\Trps := (u^\infty_1, u^\infty_2)^\Trps$ is a stationary solution of the equation~\eqref{eq:UnivAppStaticLapproxInproof}.  Consider a solution $U(t; U^0, \mathcal{I}):= (u_1(t), u_2(t))^\Trps$ of the same equation with initial condition $U^0$ and input $\mathcal{I} \in B(0,\rho_K)$. For some $\kappa>0$ to be fixed later, $\overline{U}(t) = (u_1(t) - u^\infty_1 , \sqrt{\kappa}(u_2(t) -u^\infty_2))$ is a solution of the equation

    \begin{equation*}
        \begin{cases}
            \partial_t \Bar{u}_1 + (\tau_1 + \nu_1(-\Delta)^{n_1})\Bar{u}_1 &=0,\\
            \partial_t \Bar{u}_2 + (\tau_2 + \nu_2(-\Delta)^{n_2})\Bar{u}_2 &= \sqrt{\kappa}(W_{21}\sigma(u_1 + \beta_1) - W_{21}\sigma(u^\infty_1 + \beta_1)).
        \end{cases}
    \end{equation*}

    Applying estimate~\eqref{est:LinearEqAbs} of Lemma~\ref{lem:LinearEqAbs}, we obtain
    \begin{equation*}
        \begin{split}
            \frac{\dd}{\dd t} \int |\overline{U}|^2\dd x + \tau_{\min}\int |\overline{U}|^2\dd x &\leq \kappa \mathrm{c}_1 \|W_{21}\|^2_{L^2 \to L^2} \|\sigma'\|_{\infty}^2 \int |u_1- u^\infty_1|^2\dd x, \\
            &\leq \kappa \mathrm{c}_1 \|W_{21}\|^2_{L^2 \to L^2} \|\sigma'\|_{\infty}^2 \int |\overline{U}|^2\dd x,
        \end{split}
    \end{equation*}
    for some constant $\mathrm{c}_1 > 0$. Taking $\kappa = \frac{\tau_{\min}}{2}(\mathrm{c}_1 \|W_{21}\|^2_{L^2 \to L^2} \|\sigma'\|_{\infty}^2 \vee 1 )^{-1}$, we obtain
    \begin{equation*}
        \frac{\dd}{\dd t} \int |\overline{U}|^2\dd x + \frac{\tau_{\min}}{2} \int |\overline{U}|^2\dd x \leq 0,
    \end{equation*}
    so that
    \begin{equation*}
        \begin{split}
            \|u_2(t) - F_{N_0,M,N_1}(\I)\|^2_{L^2} &\leq \kappa^{-1}\|u_1(t) - \widetilde{W}^\In_1 \I\|^2_{L^2} + \|u_2(t) - F_{N_0,M,N_1}(\I)\|^2_{L^2},\\
            &\leq (\kappa^{-1}\vee 1) (\|u_1(0) - \widetilde{W}^\In_1 \I\|^2_{L^2} + \|u_2(0) - F_{N_0,M,N_1}(\I)\|^2_{L^2})e^{-\frac{\tau_{\min}}{2}t}.
        \end{split}
    \end{equation*}
    Finally, using that $\I \in B(0,\rho_K)$ and $F_{N_0,M,N_1}(\I)$ is bounded together with the estimates~\eqref{est:ApproximationSatic0to1}--\eqref{est:ApproximationSatic1to2} we obtain the required result.
\end{proof}

The end of this subsection is devoted to proving the dynamical approximation theorem, stated in Theorem~\ref{thm:DynUniversalApprox}. The strategy consists of successively approximating the target equation, to link it to the structure of the equation~\eqref{eq:AbstractModel}. The following lemma makes it possible to compare the trajectories of two successive approximations.

\begin{lem}
    \label{lem:AbstractLipchits}
    Let $\vn \geq 0$, $\vnu > 0$, $\vtau > 0$, $\vn \geq \vs$, and $G^i \colon \dL^2 \to \dH^{-\vs}$, Lipschitz with constant $L_{G^i}$ and bounded with constant $M_{G^i}$ with $i = 1,2$. Consider the unique solutions $U^1, U^2 \in C_t (\dL^2) \cap L^2_{t,loc}(\dH^{\vn})$ of the equations 
    \begin{equation}
        \tag{$Eq^i$}
        \label{eq:AbstractLipchits}
        \partial_t U^i + (\vtau + \vnu (-\Delta)^{\vn}) U^i = G^i(U^i).
    \end{equation}
    Then, any solution of the equation~\eqref{eq:AbstractLipchits} starting in the $\dL^2$ ball of radius $\rho^i = \sqrt{c_1/\tau_{\min}} M_{G^i}$ stays inside this ball for all times, where $c_1$ is the constant from Lemma~\ref{lem:LinearEqAbs}.

    Furthermore, suppose that
    \begin{equation*}
        \|G^1(U) - G^2(U)\|_{\dH^{-\vs}} \leq \eta, \text{ for all } U\in \dL^2 \text{ with } \|U\|_{\dL^2} \leq \rho^1 \vee \rho^2.
    \end{equation*}
    Then, any two solutions with initial condition inside the respective balls of radius $\rho^i$ satisfy
    \begin{equation}
        \label{est:DiffTwosolutions}
        \sup_{t \in [0,T]} \|U^1(t)- U^2(t)\|_{\dL^2} \leq e^{c_1 L^2_{G^1}T/2}\|U^1(0)- U^2(0)\|_{\dL^2} + \eta \frac{\sqrt{\exp (c_1 L^2_{G^1}T) - 1}}{\sqrt{c_1}L_{G^1}}.
    \end{equation}
\end{lem}

The existence of such solutions follows from the same Picard--Lindel\"of strategy as in Theorem~\ref{thm:WPSemigroup}. The estimates follow from estimate~\eqref{est:LinearEqAbs} and the Gr\"onwall Lemma. We now prove Theorem~\ref{thm:DynUniversalApprox}.

\begin{proof}[Proof of Theorem~\ref{thm:DynUniversalApprox}]
    In the following, when we encounter a Lipschitz function $G\colon \dL^2 \to \dH^{-s}$, we denote by $L_G$ its Lipschitz constant. We here denote the linear operator $(\tau + \nu(-\Delta)^n)$ by $A$ and we choose $\sigma : \dR \to \dR$ to be the sigmoid $\sigma(r) := 4((1+\exp(-r))^{-1} - 1/2)$, such that $\sigma(0) = 0$ and $\sigma'(0) = 1$. 
    
    Consider the Fourier--Hilbert orthonormal basis $\{e_n\}_{n\in\dN}$ of $L^2(\Omega)$. By construction we also have that $A e_n = \lambda_n e_n$ for some $\lambda_n >0$ for all $n \geq 0 $.
    For $\Lambda>0$, we introduce the finite-dimensional subspace $E_\Lambda$ of $L^2$ defined as
    \begin{equation*}
        E_\Lambda := \Span \{e_n, \text{ such that } \|(1-\Delta) e_n\|_{L^2} \leq \Lambda, n \in \dN \}.
    \end{equation*}

    Set $N = \dim E_\Lambda$, let $P_{N} \colon L^2 \to L^2$ denote the orthogonal projection onto $E_\Lambda$, and set $Q_N := (I-P_N) \colon L^2 \to L^2$, the projection onto the orthogonal complement of $E_\Lambda$.

    From the compactness of $K$ in $V$, there exists $\rho_2 >0$ such that $\|\I\|_V \leq \rho_2$ for all $\I \in K$, and for any $\eta_1 > 0$ there exists a finite-dimensional subspace $E^\In_{\eta_1}$ of $V$ with basis $\{v_1, \cdots, v_{N^\In_{\eta_1}}\}$ and projection operator $P^\In_{\eta_1} \colon V\to V$, such that $\|P^\In_{\eta_1} \I - \I\|_V \leq \eta_1$ for all $\I \in K$.

    Now, let $\varepsilon >0$, $u^0 \in H^\gamma$ with $\|u^0\|_{H^\gamma} \leq \rho$ and $\I \in K$. We first consider the Galerkin Approximation
    \begin{equation}
        \label{eq:DATGalerkin}
        \partial_t u_N + A u_N = P_N R (u_N, P^\In_{\eta_1} \I),
    \end{equation}
    with initial condition $u_N(0) = P_N u^0$. 
    
    We note that by assumption on $R$, as $s < n$ and the above, we have for all $u \in L^2$ and $ \I \in K$
    \begin{align*}
         \|P_N R (u, P^\In_{\eta_1} \I) - R (u,\I)\|_{H^{-n}} & \leq \|Q_N R(u, \I)\|_{H^{-n}} +  \|R (u, P^\In_{\eta_1} \I) - R (u,\I)\|_{H^{-n}}\\
         &\leq \Lambda^{s-n} \|R(u, \I)\|_{H^{-s}} + L_R \eta_1.
    \end{align*}

    From the regularity of the initial condition, we have $\|u^0 - P_N u^0 \|_{L^2} \leq \Lambda^{-\gamma} \|u^0\|_{H^{\gamma}}$. Using Lemma~\ref{lem:AbstractLipchits}, we then obtain the following estimate between the solution of the equation~\eqref{eq:GeneralEquationUnivAppDyn} and its Galerkin approximation~\eqref{eq:DATGalerkin}
    \begin{equation}
        \label{est:DUATeps1}
        \sup_{t\in [0,T]} \|u(t) - u_N(t)\|_{L^2} \leq  \left(e^{c_1 L^2_{R}T/2} \rho + \frac{\sqrt{\exp (c_1 L^2_{R}T) - 1}}{\sqrt{c_1}L_{R}} C_R\vee L_R \right) (\Lambda^{-((n-s)\wedge \gamma)}+\eta_1)
    \end{equation}
    where $C_R = \sup_{u\in L^2, \I \in K}\|R(u, \I)\|_{H^{-s}}$. We then choose $\Lambda$ sufficiently large and $\eta_1>0$ sufficiently small such that the previous is bounded by $\varepsilon/3$.

    We now leverage Cybenko's Universal Approximation Theorem~\cite{cybenko1989approximation} to approximate the equation~\eqref{eq:DATGalerkin}. To this end, let $\Pi \colon \dR^N \to E_\Lambda$, $\Pi^{-1} \colon E_\Lambda \to \dR^N$, defined as $\Pi x = \sum_{n = 1}^N x_n e_n$ and $\Pi^{-1} u = (\langle e_1, u\rangle, \cdots, \langle e_N, u\rangle)^\Trps$, and similarly let $\Pi_{\In,\eta_1} \colon \dR^{N^\In_{\eta_1}} \to E^\In_{\eta_1}$ and $\Pi_{\In,\eta_1}^{-1} \colon E^\In_{\eta_1}\to \dR^{N^\In_{\eta_1}}$ be a coordinate isomorphism and its coordinate map of the finite-dimensional vector space $E^\In_{\eta_1}$, for the basis $\{v_1, \cdots, v_{N^\In_{\eta_1}}\}$. Let $f_{N} : \dR^N \times \dR^{N^\In_{\eta_1}}\to \dR^N$ be defined as $f_N (x, v) = \Pi^{-1}P_N R(\Pi x, \Pi_{\In,\eta_1} v)$. From the Universal Approximation theorem, for $\eta_2>0$, there exist $M \geq 1$, $A \in \mathcal{M}_{N,M}(\dR)$, $B^1 \in \mathcal{M}_{M,N}(\dR)$, $B^2 \in \mathcal{M}_{M,N^\In_{\eta_1}}(\dR)$ and $\theta \in \dR^M$ such that $f_{N,M} \colon \dR^N \times \dR^{N^\In_{\eta_1}} \to \dR^N$ defined as $f_{N,M}(x,v) = A \sigma(B^1 x + B^2 v + \theta)$ satisfies
    \begin{equation*}
        \sup_{\substack{x \in \dR^N,\; v \in \dR^{N^\In_{\eta_1}}\\
                    \|x\| \leq \rho,\; \|v\| \leq \rho_2}} \|f_N(x,v) - f_{N,M}(x,v)\| \leq \eta_2.
    \end{equation*}
    
    We now construct the following functional lifting. Consider a family of smooth bump functions $\varphi_1, \cdots \varphi_M \in H^{2n}$ associated with disjoint sub-rectangular domains $\Omega_1, \cdots, \Omega_M$ of $\Omega$, with nonzero constant volume, $|\Omega_m| = |\Omega_1| > 0$ for all $m = 1, \cdots, M$, such that, $\varphi_m \equiv \delta_{m \ell}$ on $\Omega_\ell$ for all $\ell, m \in \{1, \cdots, M\}$. That is, the bump functions are constant equal to one on their associated sub-rectangular domain, and zero on the others. This allows us to define the linear operators $W_{11}^0 \colon L^2 \to H^{2n}$, $W_{21} \colon L^2 \to H^{2n}$, $\widetilde{W}^\In_1 \colon V \to H^{2n}$ and $\beta_1 \in L^2$ as
    \begin{align*}
         W_{11}^0  u = \sum_{n = 1}^N \sum_{m = 1}^M \varphi_{m} B^1_{nm} \langle e_n, u\rangle, & \; W_{21}  u = \sum_{n = 1}^N \sum_{m = 1}^M e_{n} \frac{A_{mn}}{|\Omega_m|}\langle \Indic_{\Omega_m}, u\rangle,\\
         \widetilde{W}^\In_1  \mathcal{I} = \sum_{k = 1}^{N^\In_{\eta_1}} \sum_{m = 1}^M \varphi_{m} B^2_{km} (\Pi_{\In,\eta_1}^{-1} P^\In_{\eta_1} \I)_k, & \; \beta_1 = \sum_{m=1}^M \theta_m \varphi_m. 
    \end{align*}
    Then by construction, the bounded Lipschitz function $R_{N,M} \colon L^2 \times V \to H^{2n}$ defined as 
        \begin{equation*}
            R_{N,M} (u,\mathcal{I}) = W_{21} \sigma (W_{11}^0 u + \widetilde{W}^\In_1 \I+ \beta_1)
        \end{equation*}
    is a functional lifting of $f_{N,M}$, that is, $f_{N,M}(x, v) = \Pi^{-1} R_{N,M}(\Pi x, \Pi_{\In,\eta_1} v)$. Thus this leads to the estimate
    \begin{equation*}
        \sup_{\substack{u \in E_\Lambda,\; \I\in K\\
                    \|u\|_{L^2} \leq \rho}} \|P_N R(u,P_{\eta_1}^\In \I) - R_{N,M}(u,\I)\|_{L^2}\leq \eta_2.
    \end{equation*}
    We now compare the solutions of
    \begin{equation}
        \label{eq:CybenkoGalerkin}
        \partial_t \Tilde{u}+ A \Tilde{u} = R_{N,M} (\Tilde{u}, \mathcal{I})
    \end{equation}
    with initial condition $\Tilde{u}(0) = u_N (0)$, with the solution of the Galerkin Approximation with Lemma~\ref{lem:AbstractLipchits} as
    \begin{equation}
        \label{est:DUATeps2}
        \sup_{t\in [0,T]} \| u_N(t)- \Tilde{u}_N(t) \|_{L^2} \leq \frac{\sqrt{\exp (c_1 L^2_{R}T) - 1}}{\sqrt{c_1}L_{R}} \eta_2.
    \end{equation}
    The previous holds, as the Lipschitz constant of the Galerkin approximation $P_N R$ is bounded by the Lipschitz constant of $R$. We also note that by construction $\Tilde{u}(t) \in E_\Lambda$ for all times. We then can choose $\eta_2>0$ to be sufficiently small so that the above is bounded by $\varepsilon/3$.

    We construct the last approximated equation by first writing the equation~\eqref{eq:CybenkoGalerkin} with the augmented state variable $W_{11}^0 \Tilde{u} + \widetilde{W}^\In_1 \I$,
    \begin{equation*}
        \begin{cases}
            \partial_t \Tilde{u}_1 + A \Tilde{u}_1 = W^0_{11}W_{21} \sigma(\Tilde{u}_1 + \beta_1) + (A W^0_{11} - W^0_{11} A) \Tilde{u}_2 + A \widetilde{W}^\In_1 \mathcal{I},\\
            \partial_t \Tilde{u}_2 + A \Tilde{u}_2 = W_{21} \sigma(\Tilde{u}_1 + \beta_1),
        \end{cases}
    \end{equation*}
    with initial condition $\Tilde{U}(0) = (\Tilde{u}_1(0), \Tilde{u}_2(0))^\Trps = (W_{11}^0 u_N(0) + \widetilde{W}^\In_1 \mathcal{I},u_N(0))^\Trps$. We note that, as $\Tilde{u}_2 \in E_\Lambda$, the previous equation stays equivalent to a Lipschitz ODE and by uniqueness $\Tilde{u}_2$ coincides with $\Tilde{u}$. We now define $W_{11} :=  W^0_{11}W_{21} \colon L^2 \to L^2$, $W^\In_1 := A \widetilde{W}^\In_1 \colon V \to L^2$. Then, using that $A$ is self-adjoint and the basis of $E_\Lambda$ is an eigenbasis of $A$, specifically we have that $\langle e_n, A u \rangle = \langle A e_n, u \rangle = \lambda_n \langle e_n , u \rangle$, so that 
    \begin{equation*}
        W^0_{12} := (A W^0_{11} - W^0_{11} A) = \sum_{n = 1}^N \sum_{m = 1}^M (A\varphi_{m} - \lambda_n \varphi_m ) B^1_{nm} \langle e_n, \cdot \rangle,
    \end{equation*}
    and $W^0_{12} \colon L^2 \to L^2$, from the smoothness of the functions $\varphi_m$. Thus, the previous system can be written as
    \begin{equation}
        \label{eq:CybenkoGalerkinAgmented}
        \begin{cases}
            \partial_t \Tilde{u}_1 + A \Tilde{u}_1 = W_{11} \sigma(\Tilde{u}_1 + \beta_1) + W^0_{12} \Tilde{u}_2 + W^\In_1 \mathcal{I},\\
            \partial_t \Tilde{u}_2 + A \Tilde{u}_2 = W_{21} \sigma(\Tilde{u}_1 + \beta_1).
        \end{cases}
    \end{equation}
    The right-hand side is a Lipschitz function from $\dL^2\times V \to \dL^2$, with a Lipschitz constant $L_{N,M}$, and the structure ensures that the solutions stay bounded. Furthermore, since $\Tilde{u}_2$ stays in $E_\Lambda$, it is in particular a function in $L^\infty$, with $\|\Tilde{u}_2\|_{L^\infty} \leq C_{N,M}$, for some $C_{N,M} > 0$. We denote by $\widetilde{U} = (\Tilde{u}_1, \Tilde{u}_2)^\Trps$ this solution.
    
    From the choice of $\sigma$, we can ensure that for all $\eta_3 > 0$, there exists $h > 0 $ sufficiently large such that
    \begin{equation*}
        \sup_{r\in [-C_{N,M},C_{N,M}]} \left|r - h\sigma\left(\tfrac{1}{h}r\right)\right| \leq \frac{\eta_3}{\|W^0_{12}\|_{L^2\to L^2} \vee 1}.
    \end{equation*}
     This ensures that for all $u \in L^\infty$ with $\|u\|_{\infty} \leq C_{N,M}$
     \begin{equation}
        \label{est:SigmaApproximateIdentity}
         \left\|W^0_{12} u - W^0_{12} h \sigma\left(\tfrac{1}{h} u\right) \right\|_{L^2} \leq \eta_3.
     \end{equation}
     Defining the constant $L^\infty$-function $\alpha_2 \equiv h^{-1}$ and the linear operator $W_{12} := h W^0_{12} \colon L^2 \to L^2$, we introduce the final approximated equation:
     \begin{equation}
        \label{eq:LDNFEq4DUAT}
         \begin{cases}
             \partial_t u_1 + A u_1 =  W_{11} \sigma(u_1 + \beta_1) + W_{12} \sigma(\alpha_2 u_2) + W^\In_1 \mathcal{I},\\
             \partial_t u_2 + A u_2 = W_{21}\sigma(u_1 + \beta_1),
         \end{cases}
     \end{equation}
     with initial condition $U_0 = (W^0_{11} P_N u^0 + \widetilde{W}^\In_1 \I, P_N u^0)^\Trps$. This last equation is a two-layer dissipative neural field equation of the form~\eqref{eq:AbstractModel}. We denote by $U = (u_1, u_2)^\Trps$ its solution. 
     
     From the estimate~\eqref{est:SigmaApproximateIdentity} and Lemma~\ref{lem:AbstractLipchits}, building on the $L^\infty$-bound of the solution for this specific initial condition, we obtain
     \begin{equation}
        \label{est:DUATeps3}
         \sup_{t\in [0,T]} \| U(t) - \widetilde{U}(t) \|_{\dL^2} \leq \frac{\sqrt{\exp (c_1 L^2_{N,M}T) - 1}}{\sqrt{c_1}L_{N,M}} \eta_3.
     \end{equation}
    Taking $\eta_3$ sufficiently small such that the previous estimate holds with an $\varepsilon/3$ bound concludes the proof.
\end{proof}

\subsection{Bifurcation Analysis of Spontaneous Cortical Activity}

In this subsection, we prove the results stated in Theorems~\ref{thm:BifStatDiag}, \ref{thm:BifStatDiagSelfAdjPerturbation} and~\ref{thm:BifHopfDiag}, concerning the potential qualitative behavior of the abstract model proposed in this paper. We begin the analysis with the following properties of the functional equation.

\begin{prop}
    \label{prop:FfunctionalProperties} Let $n \geq 1$, $\nu > 0 $, $\tau >0$. Under the assumptions~\eqref{as:SCA-AssW}, \eqref{as:SCA-AssW2} and \eqref{as:SCA-AssSigma} the functional $F\colon H^{2n}\times \dR \to L^2$ defined in~\eqref{def:SCAFunctional} is in the space $F \in C^{k}(H^{2n}\times \dR, L^2)$. $F$ is a nonlinear Fredholm operator of index 0, that is, $\dim \ker (D_uF(u,\lambda)) = \dim \coker (D_uF(u,\lambda))$ for all $u, \lambda$. Finally, $-D_uF(0,\lambda)$ generates an analytic semigroup for all $\lambda$.
\end{prop}

\begin{proof}
    The continuity of $F \in C(H^{2n}\times\dR, L^2)$ and $D_uF \in C(H^{2n}\times\dR, L(H^{2n}, L^2))$ follows from Morrey's embedding theorem $H^{s}(\Torus) \subset C^0_b(\Torus)$ with $s >1/2$. The linearized operator is given by
    \begin{equation*}
        D_uF(u, \lambda)[\phi] = (\nu(-\Delta)^n + \tau) \phi - \lambda W (\sigma'(u) \phi),
    \end{equation*}
    as the operator $(\nu(-\Delta)^n + \tau) \colon H^{2n} \to L^2$ is Fredholm of index zero from elliptic theory and as $W$ is an operator of order $s$ with $s\leq n$, we have that $W \in L(H^{2n}, H^{2n-s})$ with $2n-s > 0$, by the Rellich--Kondrachov theorem, as the embedding $H^{2n-s} \subset \subset L^2$ is compact, $W$ is a compact operator from $H^{2n}$ to $L^2$, and thus $ - \lambda W (\sigma'(u) \cdot )$ is a compact perturbation of a Fredholm operator, so that $D_uF(u,\lambda)$ is Fredholm of index zero. Finally, the analytic semigroup property of $-D_u F(0,\lambda)$ follows from~\cite[Thm~1.3.4]{henry2006geometric}, as it is a sectorial operator by~\cite[Thm~1.3.2]{henry2006geometric}.
\end{proof}

We now prove Theorem~\ref{thm:BifStatDiag} regarding the existence of nontrivial stationary solutions of the equation~\eqref{eq:StationnarySol1D1L} for diagonal operators $W$.

\begin{proof}[Proof of Theorem~\ref{thm:BifStatDiag}]
    We first note that as $W$ is diagonal and commutes with the reflection operator,  the coefficients $\hat{W}_\ell$ are real and $\hat{W}_\ell = \hat{W}_{-\ell}$ for all $\ell \in \dZ$, and thus the operators $W$ and $A - \lambda W$ are self-adjoint. We then define the following function spaces, $X = H^{2n}\cap \{ u (x) = u(-x)\}$ continuously embedded into $Z = L^2 \cap \{ u (x) = u(-x)\}$ and $Z$ is equipped with the $L^2$ scalar product. We define the equation functional $F \colon X\times \dR \to Z$, as
    \begin{equation*}
        F(u,\lambda ) = A^n_\tau u -\lambda W \sigma(u),
    \end{equation*}
     and its Fréchet derivative at $(0,\lambda)$ is given by
    \begin{equation*}
        D_u F (0,\lambda) = A^n_\tau - \lambda W : H^{2n} \to L^2.
    \end{equation*}
    Let $k\geq 1$ be an exceptional index and define $\lambda^k_0 = \hat{A}_k/\hat{W}_k \in \dR$. Using that the linearized operator is diagonal in Fourier, the kernel equation can be written as
    \begin{equation*}
        (\hat{A}_\ell-\lambda^k_0 \hat{W}_\ell)\hat{u}_\ell = 0, \forall \ell \in \dZ.
    \end{equation*}
    Since $k$ is an exceptional index, we obtain that,
    \begin{equation*}
        \hat{u}_\ell = 0, \forall \ell \in \dZ, |\ell| \neq k,
    \end{equation*}
    and the kernel of the linearized operator in $X$ is given by $\Span \{ \phi\}$, with $\phi(x) = \frac{1}{2\sqrt{\pi}}(e^{ikx} + e^{-ikx})$. As the linearized operator is self-adjoint, the kernel and the cokernel coincide. We define the projection operator $P \colon X \to X$ onto the kernel as $P = \langle \phi, \, \cdot \,\rangle \phi$. We then directly apply the Crandall--Rabinowitz theorem, see \textit{e.g.}~\cite[Thm~I.5.1]{kielhofer2012bifurcation}, by checking that $D^2_{u,\lambda}F(0,\lambda^k)[\phi] \notin R(D_uF(0,\lambda^k))$. This follows directly from
    \begin{equation*}
        \langle D^2_{u\lambda}F(0,\lambda^k)[\phi], \phi \rangle = - \langle W \phi, \phi \rangle = -\hat{W}_k \neq 0.
    \end{equation*}

    The family of bifurcating solutions is obtained by translation invariance. We conclude the proof by applying the formulas of~\cite[pp.~18--19]{kielhofer2012bifurcation} to obtain the Taylor expansion of $\lambda(s)$. The first order expansion of the function $u^k_{(\theta, s)}$ is already given by the previously mentioned result.

    The formulas are
    \begin{equation*}
        \dot{\lambda}(0) = - \frac{1}{2}\frac{\langle D^2_{uu}F(0,\lambda^k)[\phi, \phi], \phi\rangle}{\langle D^2_{u\lambda}F(0,\lambda^k)[\phi], \phi\rangle},
    \end{equation*}
    and
    \begin{equation*}
        \ddot{\lambda}(0) = - \frac{1}{3}\frac{\langle \mathrm{T}, \phi\rangle}{\langle D^2_{u\lambda}F(0,\lambda^k)[\phi], \phi\rangle},
    \end{equation*}
    where $\mathrm{T} = P D^3_{uuu}F(0,\lambda^k)[\phi, \phi, \phi] - 3 P D^2_{uu} F(0,\lambda^k)[\phi, (I-P) (D_uF(0,\lambda^k)^{-1}|_{\phi^\perp}(I-P)D^2_{uu}F(0,\lambda^k)[\phi,\phi]]$. Given that $D^2_{uu}F(0,\lambda^k)[\phi, \phi] = -\lambda^k\sigma_2 W (\phi^2)$ and $D^3_{uuu}F(0,\lambda^k)[\phi, \phi, \phi] = -\lambda^k\sigma_3 W (\phi^3)$, through explicit computations we obtain the stated result.

    The principle of exchange of stability for pitchfork bifurcation gives the sign of the critical eigenvalue of the linearization along the bifurcating branch. Finally, since the operator is sectorial, nonlinear (in)stability follows from~\cite[Thm~5.1.1]{henry2006geometric}.
\end{proof}

We now prove an extension of the previous result for self-adjoint perturbations of diagonal operators.

\begin{proof}[Proof of Theorem~\ref{thm:BifStatDiagSelfAdjPerturbation}]
    We begin with the eigenvalue perturbation problem. We recall the spaces $X = H^{2n}\cap \{ u (x) = u(-x)\}$ and $Z = L^2 \cap \{ u (x) = u(-x)\}$, such that $X$ is continuously embedded into $Z$. Define the function $G: X\times \dR \times \dR \to Z\times \dR$ as

    \begin{equation*}
        G(\phi, \lambda, \omega) = \begin{pmatrix}
            (A-\lambda W(\omega)) \phi\\
            \|\phi\|^2 - 1
        \end{pmatrix}.
    \end{equation*}
    
    From the assumptions on $W_0$, we have
    \begin{equation*}
        G(\phi_0, \lambda_0, 0) = 0,
    \end{equation*}
    with $\phi_0 = \frac{1}{\sqrt{\pi}}\cos(k \, \cdot\,)$, $\lambda_0 = \hat{A}_k / \hat{W}_{0;k}$. We also have the decomposition $X = \Span \{\phi_0\} \oplus^\perp X_1$ and $Z = \Span \{\phi_0\} \oplus^\perp Z_1$, such that
    $(A- \lambda_0 W(0)) : X_1 \to Z_1$ is a Banach isomorphism. 
    %Using that $W(\omega)$ is a relatively compact perturbation of $A$, and that $A$ is of compact resolvent, the spectrum of $(A- \lambda_0 W(\omega))$ is discrete with no accumulation point on a compact domain. 
    One easily verifies that $G \in C^1(X\times \dR \times \dR, Z\times \dR)$. The Jacobian of $G$ with respect to $\phi,\lambda$ is given by
    \begin{equation*}
        D_1 G (\phi_0, \lambda_0, \omega) = \begin{pmatrix}
                                            A-\lambda_0 W(\omega) & -W(\omega) \phi_0\\
                                            2\langle \phi_0, \,\cdot \,\rangle & 0
                                        \end{pmatrix}.
    \end{equation*}

    We then obtain that $D_1 G(\phi_0, \lambda_0, 0) : X \times \dR \to Z \times \dR$  is a Banach isomorphism, with its inverse given by
    \begin{equation*}
        (D_1 G(\phi_0, \lambda_0, 0))^{-1} = \begin{pmatrix}
                                                (A-\lambda_0 W_0)|^{-1}_{X_1}(\;\cdot - \frac{\langle \;\cdot\;, \phi_0\rangle}{\langle \phi_0 , W_0 \phi_0\rangle} W_0 \phi_0) & \frac{1}{2}\phi_0\\
                                                - \frac{\langle \;\cdot\;, \phi_0\rangle}{\langle \phi_0 , W_0 \phi_0\rangle} & 0
                                        \end{pmatrix}.
    \end{equation*}
    Applying the Banach implicit function theorem, we obtain the existence of a local curve $\omega \mapsto (\lambda(\omega), \phi_0(\omega))$, and space decompositions $X = X_0(\omega) \oplus^\perp X_1(\omega)$, $Z = X_0(\omega) \oplus^\perp Z_1(\omega)$, where we denote $X_0(\omega) = \Span\{\phi_0(\omega)\}$, and associated projectors $P(\omega):X\to X_0(\omega)$ such that
    \begin{equation}
        \ker ((A-\lambda(\omega)W(\omega)) = \coker  ((A-\lambda(\omega)W(\omega)) = X_0(\omega).
    \end{equation}
    The kernel and cokernel coincide as the operator is self-adjoint. Furthermore, we have the following expansions, $\lambda(\omega) = \lambda_0 + \omega \lambda_1 + o(\omega)$, $\phi_0(\omega) = \phi_0 + \omega \phi_1 + o(\omega)$, where
    \begin{align*}
        \lambda_1 &= - \lambda_0 \frac{\langle \phi_0, W_1 \phi_0\rangle}{\langle \phi_0, W_0 \phi_0\rangle},\\
        \phi_1 & = (A-\lambda_0 W_0)|^{-1}_{X_1}(\lambda_0 W_1 \phi_0 + \lambda_1 W_0 \phi_0).
    \end{align*}
    We then define the following equation functional $F : X \times \dR \times \dR \to Z$ as
    \begin{equation*}
        F(u, \lambda, \omega) = A u -\lambda W(\omega)\sigma(u),
    \end{equation*}
    and we introduce the following parametric Lyapunov--Schmidt reduction of the problem $F(u,\lambda, \omega)= 0$, around the solution $(0,\lambda_0, 0)$, by writing the following space decomposition $X = X_0 \oplus^\perp X_1$, with $X_0 = \Span \{ \phi_0\}$ and $Z = X_0 \oplus^\perp Z_1$, the equation can be written as
    \begin{align*}
        (I - P) F(s \phi_0 + g, \lambda, \omega) = 0 \text{ in } Z_1,\\
        P F(s \phi_0 + g, \lambda, \omega) = 0  \text{ in } X_0,
    \end{align*}
    where $g \in X_1,\lambda\in \dR$. Applying the Lyapunov--Schmidt method, we obtain, from the Banach implicit function theorem, the existence of a function $g : \dR^3 \to X_1$ that reduces the problem to the finite-dimensional second equation. Furthermore, we have that $D_{\lambda}g(0, 0, \omega) = 0 , D_{s}g( 0, 0, \omega) = 0$, see \textit{e.g.}~\cite[Cor.~I.2.4]{kielhofer2012bifurcation}. We then define the function $\Phi : \dR^3 \to \dR $ as 
    \begin{equation*}
        \Phi(\lambda, \omega, s) = \langle \phi_0 , F(s \phi_0 + g(\lambda, s, \omega), \lambda, \omega) \rangle = \int_0^1 \langle \phi_0 , D_u F(s t \phi_0 + g(\lambda, s t, \omega), \lambda, \omega)[s\phi_0 + s\partial_s g(\lambda, s t, \omega)] \rangle \dd t.
    \end{equation*}
    From the last identity, we define the rescaled equation $\Tilde{\Phi}(\lambda, \omega, s) = \Phi(\lambda, \omega, s)/s$, and we observe that
    \begin{align*}
        \partial_\lambda \Tilde{\Phi}(\lambda_0, 0, 0) &= \int_0^1 \langle \phi_0 , D^2_{u\lambda} F(0, \lambda_0, 0)[\phi_0 + \partial_s g(\lambda_0, 0, 0)] \rangle\\
        &\hspace{3em}+ \langle \phi_0, D^2_{uu} F(0, \lambda_0, 0)[\phi_0 + \partial_s g(\lambda_0, 0, 0), \partial_\lambda g(\lambda_0, 0, 0)] \rangle \dd t,\\
        &=  \langle \phi_0 , D^2_{u\lambda} F(0, \lambda_0, 0)[\phi_0] \rangle,
    \end{align*}
    the last term is nonzero from the same argument as in the proof of Theorem~\ref{thm:BifStatDiag}, thus giving the transversality condition. The implicit function theorem gives the existence of a surface solution $(s,\omega) \mapsto \lambda(s,\omega)$. Finally, by differentiating the equation
    \begin{equation*}
        F(s\phi_0 + g(\lambda(s,\omega), s, \omega), \lambda(s, \omega), 
        \omega) = 0, 
    \end{equation*}
    with respect to $s$, and evaluating at $s=0$, we obtain
    \begin{equation*}
        0 = D_u F(0, \lambda(0, \omega), \omega)[\phi_0 + \partial_s g(\lambda(0,\omega), 0, \omega)] + D_\lambda F(0, \lambda(0, \omega), \omega)[\partial_\lambda g (\lambda(0,\omega), 0, \omega)].
    \end{equation*}
    where we used that $g(\lambda, 0, \omega) = 0$. As $\partial_\lambda g (\lambda(0,\omega), 0, \omega) = 0$, this implies that the kernel of $D_u F(0, \lambda(0, \omega), \omega)$ contains $\phi_0 + \partial_s g(\lambda(0,\omega), 0, \omega)$. Locally in $\omega$, the kernel is exactly given by $\phi_0(\omega)$ along $\lambda(\omega)$. Up to a reparameterization and shrinking of the domain of definition, the two curves coincide. This gives the expansion of $\lambda(s,\omega)$ in $\omega$ at first order; the expansion in $s$ is similar to the one obtained in the proof of Theorem~\ref{thm:BifStatDiag}, and $\partial_s g(\lambda_0, 0, \omega) = \omega \phi_1$.
\end{proof}

We conclude this section with the proof of Theorem~\ref{thm:BifHopfDiag} on the existence of traveling wave solutions.

\begin{proof}[Proof of Theorem~\ref{thm:BifHopfDiag}]
    We directly verify the assumptions of~\cite[Thm~I.8.2]{kielhofer2012bifurcation}. The assumptions on $F$ follow from Proposition~\ref{prop:FfunctionalProperties}. We only check the spectral properties at a \textbf{Hopf exceptional point}. We have that
    \begin{equation*}
        D_u F\left(0,(\Real (\mathrm{c}_k))^{-1}\right) e_k = -  \mathrm{i}  \frac{\Imag (\hat{W}_k) \langle k \rangle^n_{\nu, \tau}}{\Real (\hat{W}_k)  } e_k,
    \end{equation*}
    the uniqueness of the value $\mathrm{c}_k$ in the sequence $(\mathrm{c}_\ell)_{\ell \in \dN}$ ensures the non-resonance condition, and the transversality condition follows directly from the linear dependence of the eigenvalues $\pm \mathrm{i}  \frac{\Imag (\hat{W}_k) \langle k \rangle^n_{\nu, \tau}}{\Real (\hat{W}_k)}$ in $\lambda$. The expansion of the solution follows from~\cite[Cor.~I.8.3]{kielhofer2012bifurcation}.
\end{proof}

%% file: InertialManifold.tex
\begin{lem}[Absorption]
    \label{lem:Absorption}
    For $\vtau >0, \vnu >0, \vn \geq 0, \vn \geq \vs$, the ball $B_{\dL^2}(0,\rho(r))$ in $\dL^2$ with radius
    \begin{equation}
        \rho(r) := 2^{\tfrac{(n_{\max}-1)}{2}}(\nu_{\min} \wedge \tau_{\min})^{-\frac{1}{2}}\tau_{\min}^{-\frac{1}{2}}(\|W\|_{\dL^2\to \dH^{-\vs}} \|\sigma\|_{\infty} + r+1),
    \end{equation}
    is absorbing for $\Psi$, uniformly with respect to inputs in $B_{\W^\In V}(0,r)$, for any $r\geq0$. 
\end{lem}

The proof directly follows from estimates~\eqref{est:ARBounded}
and~\eqref{est:LinearEqAbs} together with Gr\"onwall's lemma.

\begin{lem}[Spectral Gap Estimate]
    \label{lem:SpectralGap}
    Consider the operator $(\vtau + \vnu (-\Delta)^{\vn}) : \dH^{2\vn} \to \dL^2$, with $\vtau > 0, \vnu > 0, \vn \geq 2$, and write $\Sigma (\vtau + \vnu (-\Delta)^{\vn}) = \{0<\lambda_1 \leq \cdots \leq \lambda_N \to \infty \}$ counting multiplicity. Let $0\leq \beta < 1$, suppose that $(1-\beta)n_{\min}>1$ and denote $\delta = 1/((1-\beta)n_{\min}-1)$. Then, for all $K_1, K_2 \geq 0$, there exists an integer $N$ such that
    \begin{equation}
        \label{est:SpectralGapCondition}
        \lambda_{N+1} > K_1
        \;\text{ and }\;
        \lambda_{N+1} - \lambda_{N} > K_2(\lambda_{N+1}^\beta + \lambda_{N}^\beta),
    \end{equation}
    together with the bound
    \begin{equation}
        \label{est:SpectrlGapEstimate}
        N \leq (c_0 d)^{1+\delta}\left(\nu_{\min}^{-(1-\beta)\delta} \vee \tau_{\max}^{-(1-\beta)\delta}\vee 1\right)\left(\tau_{\max}^{1/n_{\min}} \vee K_1^{1/n_{\min}} \vee K_2^\delta\right),
    \end{equation}
    where $c_0\geq 1$ is an absolute constant independent of the parameters.
\end{lem}

\begin{proof}
    We begin by decomposing the spectrum of the operator as
    \begin{equation*}
        \Lambda := \Sigma (\vtau + \vnu (-\Delta)^{\vn}) = \bigcup_{\ell=1}^d\Sigma (\tau_\ell + \nu_\ell (-\Delta)^{n_\ell}),
    \end{equation*}
    with $\Sigma (\tau_\ell + \nu_\ell (-\Delta)^{n_\ell}) = \{\tau_\ell + \nu_\ell (4\pi^2|\xi|^2)^{n_\ell} , \xi \in \dZ^2\}$. We set $N(T) := \# (\Lambda \cap [0,T])$ and $N^\ell(T) := \# (\Sigma (\tau_\ell + \nu_\ell (-\Delta)^{n_\ell}) \cap [0,T])$, both counting multiplicity. So that, $N(T) = \sum_{\ell = 1}^d N^\ell(T)$. From Weyl's law we obtain the upper bound, 
    \begin{equation*}
        N^\ell(T) \leq \pi \left(\tfrac{1}{2\pi}\left(\frac{T-\tau_\ell}{\nu_\ell}\right)^{1/2 n_\ell}+ \frac{\sqrt{2}}{2}\right)^2, \; \forall T\geq \tau_\ell.
    \end{equation*}
    This can be obtained similarly with Gauss's argument for the Gauss circle problem, as
    \begin{equation*}
        N_0(R) := \#\{\xi \in \dZ^2, |\xi|\leq R\} \leq \pi (R+1/\sqrt{2})^2, \text{ for all } R\geq 0.
    \end{equation*}
    We then derive the following upper bound on $N(T)$, for all $T \geq \tau_{\max}$
    \begin{equation*}
        N(T) \leq d\pi \left(\nu_{\min}^{-1/n_{\min}}\vee \tau_{\max}^{-1/n_{\min}}\vee 1\right)T^{1/n_{\min}} =: c_1 T^{1/n_{\min}}.
    \end{equation*}
    Now, take $T\geq \tau_{\max} \vee 1$ and consider
    \begin{equation*}
        n = \# (\Lambda \cap [T,2T]) \leq N(2T) \leq c_1 (2T)^{1/n_{\min}},
    \end{equation*}
    and as $c_1 \geq 1$, we can derive $n+1 < 3c_1 T^{1/n_{\min}}$. These $n$ eigenvalues subdivide $[T,2T]$ into at most $n+1$ intervals with disjoint interiors. One of these sub-intervals, say $[u,v]$, has a length bigger than the average length, 
    \begin{equation*}
        v-u \geq \frac{T}{n+1} > \frac{1}{3c_1}T^\gamma,
    \end{equation*}
    where $\gamma = 1-1/n_{\min} < 1$. We denote $\lambda_N = \max \{\lambda \in \Lambda, \lambda \leq u\}$ and $\lambda_{N+1} = \min \{\lambda \in \Lambda, \lambda \geq v\}$. By definition, $\lambda_N<\lambda_{N+1}$ are successive eigenvalues without other elements of $\Lambda$ in between. We also have by construction $\lambda_{N+1} \geq T$, and from the previous estimate on the length of $[u,v]$ we have,
    \begin{equation*}
        \lambda_{N+1} - \lambda_N > \frac{1}{3c_1}T^\gamma.
    \end{equation*}
    We now consider two cases. Either $\lambda_{N+1}\leq 4T$, which leads to
    \begin{equation*}
        \frac{\lambda_{N+1} - \lambda_N}{\lambda_{N+1}^\beta+\lambda_N^\beta} > \frac{\lambda_{N+1} - \lambda_N}{2\lambda_{N+1}^\beta} > \frac{1}{24c_1}T^{\gamma-\beta}.
    \end{equation*}
    Or $\lambda_{N+1}>4T$, which leads to $\lambda_{N+1} >2T \geq \lambda_N$, so that $\lambda_{N+1} - \lambda_N > \lambda_{N+1}/2$ and
    \begin{equation*}
        \frac{\lambda_{N+1} - \lambda_N}{\lambda_{N+1}^\beta+\lambda_N^\beta} > \tfrac{1}{4}\lambda^{1-\beta} > \tfrac{1}{4}(4T)^{1-\beta} > \frac{1}{4c_1}T^{1-\beta}.
    \end{equation*}
    So in both cases, we have the bound,
    \begin{equation*}
        \frac{\lambda_{N+1} - \lambda_N}{\lambda_{N+1}^\beta+\lambda_N^\beta} > \frac{1}{24c_1}T^{\gamma-\beta}.
    \end{equation*}
    Finally, for $K_1, K_2 \geq 0$, take $T = 1 \vee \tau_{\max} \vee K_1 \vee (24 c_1 K_2)^{1/(\gamma-\beta)}$. The above ensures the existence of $N\geq 1$, such that $\lambda_N< \lambda_{N+1}$ satisfy the spectral gap stated in the statement. The stated estimate is obtained by substituting this value of $T$ into the estimate $N \leq c_1 (2T)^{1/n_{\min}}$.
    
\end{proof}

We now prove Theorem~\ref{thm:InertialManifold} by applying~\cite[Ch.~VIII, Thm~3.1]{temam2012infinite} for each fixed $F \in \W^\In B_V(0,r)$. Under appropriate assumptions, the Lyapunov--Perron functional is defined on the same eigenspace for all $F$, so that we have a function $\Phi : \P_N\dL^2 \times B_{r} \to (I-\P_N)\dL^2$, and we only need to prove the Lipschitz property in the second variable.

\begin{proof}[Proof of Theorem~\ref{thm:InertialManifold}]
    
    We fix a ball of radius $r_\In$ in $V$ that we denote by $B_{r}:= \W^\In B_V(0,r_\In) \subset B_{\W^\In V}(0,r) $, with $r=  \|\W^\In\|_{V\to \dH^{-\vs}} r_{\In}$ and consider the equation
    \begin{equation}
        \label{eq:AbstractModelRFin}
        \partial_t U + AU = R(U,F_\In) \hspace{1em} \text{ in } (0,+\infty) \times \Omega,
    \end{equation}
    for $F_\In \in B_{r}$, with $R$ as defined in Lemma~\ref{lem:LipNBoundedR} and $A = (\vtau + \vnu (-\Delta)\vn)$.

     We then define $\gamma= \max_{\ell = 1}^d \left(\frac{s_\ell}{2 n_\ell}\right)$. By assumption either $n_{\min} = 2$ and $\vs<\vn$ or $n \geq 3$ and $\vs\leq \vn$ so that we have in both cases
    \begin{equation}
        s < 2(1-1/n_{\min}) n,
    \end{equation}
    thus $0 \leq \gamma\leq \tfrac{1}{2}$ and we obtain that $\gamma<1-1/n_{\min}$, or equivalently, $(1-\gamma)n_{\min} -1 >0$.  Then Lemma~\ref{lem:SpectralGap} applies for arbitrarily large $K_1, K_2 \geq 0$. We also have, by definition that $\vs \leq 2 \gamma\vn$, and so Lemma~\ref{lem:FractionalResolventEstimateA} applies and $\|A^{-\gamma}\|_{\dH^{-\vs}\to\dL^2}$ is bounded. Together with Lemma~\ref{lem:LipNBoundedR}, this leads to,
    \begin{equation}
        \label{est:ARLipschitz}
        \|A^{-\gamma}(R(U_1, F) - R(U_2, F))\|_{\dL^2} \leq 2^{\gamma(n_{\max}-1)}(\tau_{\min}\wedge\nu_{\min})^{-\gamma} L_1 \|U_1 - U_2\|_{\dL^2},
    \end{equation}
    for all $U_1, U_2 \in \dL^2, F\in \W^\In V$, where $L_1 = \|\W\|_{\dL^2\to\dH^{-\vs}}\|\sigma'\|_{\infty}\|\upalpha\|_{\dL^\infty}$ and
    \begin{equation}
        \label{est:ARBounded}
        \|A^{-\gamma}R(U, F)\|_{\dL^2} \leq 2^{\gamma(n_{\max}-1)}(\tau_{\min}\wedge\nu_{\min})^{-\gamma} (\|\W\|_{\dL^2\to\dH^{-\vs}}\|\sigma\|_{\infty}+r),
    \end{equation}
    for all  $U\in \dL^2, F\in B_{r}$.

    Defining
    \begin{equation*}
        M_1(r) := \sup_{F\in B_{r}} \sup_{U \in \dL^2} \|A^{-\gamma}R(U, F)\|_{\dL^2}
    \end{equation*}
    and
    \begin{equation*}
        M_2 := 2^{n_{\max}/2 +1}(\tau_{\min}\wedge\nu_{\min}\wedge 1)^{-\gamma}(\tau_{\max}\vee L_1\vee 1),
    \end{equation*}
    we can ensure that
    \begin{equation}
        \label{est:ConstantM2BoundsTemam}
        M_2 \geq 2\frac{M_1(r)}{\rho(r)}+2^{\gamma(n_{\max}-1)}(\tau_{\min}\wedge\nu_{\min})^{-\gamma}L_1.
    \end{equation}

    Setting $K_1 = M_2^2(20 + 4\kappa)^{2} =: c_1 M^2_2 $  and  $K_2 = 72 M_2=: c_2 M_2$ in Lemma~\ref{lem:SpectralGap},  with $\kappa$ the absolute constant from~\cite[Ch.~VIII]{temam2012infinite}, denoted $\kappa(\gamma)$, which we can bound by $\kappa(\gamma) \leq 2$ independently of $\gamma$, we obtain that there exists an integer $N$ satisfying
    \begin{equation*}
        N \leq (c_0d)^{1+\delta_0}(\nu_{\min}\wedge\tau_{\min}\wedge 1)^{-(1-\gamma)\delta_0}(\tau_{\max}^{1/n_{\min}}\vee (c_1 M^2_2)^{1/n_{\min}}\vee (c_2 M_2)^\delta_0),
    \end{equation*}
    where $\delta_0 = 1/((1-\gamma)n_{\min}-1)$, such that the consecutive eigenvalues $\lambda_{N+1} =: \Lambda$ and $\lambda_N =: \lambda$ of operator $A$ satisfy
    \begin{equation}
        \label{hyp:InertialManifoldSpecGrowth}
        \lambda_{N+1} > M_2^2\left(\frac{1+l}{l}+ 4\kappa + 11\right)^{2}
    \end{equation}
    and 
    \begin{equation}
        \label{hyp:InertialManifoldSpecGap}
        \lambda_{N+1} - \lambda_N > 8 \frac{1+l}{l}M_2 (\lambda^\gamma_{N+1} + \lambda^\gamma_N)
    \end{equation}
    with $l= 1/8$.

    Furthermore, since $(1-\gamma)n_{\min}-1>0$ and $n_{\min} \geq 2$ we have that $n_{\min}-4\gamma >0$ and 
    \begin{equation}
        \delta := 2/(n_{\min}-4\gamma) \geq \delta_0 \vee 2/n_{\min}
    \end{equation}
    and using the definition of $M_2$, we can rewrite the upper bound on $N$ as
    \begin{equation}
        \label{est:DimensionalEstimateinProof}
         N \leq (\kappa_0 d)^{1+\delta}2^{\frac{n_{\max}}{2}\delta}(\tau_{\min} \wedge\nu_{\min}\wedge 1)^{-\delta} (\tau_{\max}\vee L_1 \vee 1)^\delta.
    \end{equation}

    Let $\P \colon \dL^2\to \dL^2$ denote the projection operator onto the eigenspace of the first $N$ eigenvalues of $A$, and set $\Q := (\mathrm{I}-\P) \colon \dL^2 \to \dL^2$. Enlarging if necessary the constant in the spectral growth condition~\eqref{hyp:InertialManifoldSpecGrowth}, estimates~\eqref{est:ARLipschitz}, \eqref{est:ARBounded}, \eqref{est:ConstantM2BoundsTemam}, \eqref{hyp:InertialManifoldSpecGrowth}, \eqref{hyp:InertialManifoldSpecGap} and Lemma~\ref{lem:Absorption} allow us to apply~\cite[Ch.~VIII, Thm~3.2]{temam2012infinite}: for any $F_\In \in B_{r}$, there exists $\Phi_{F_\In} \colon \P \dL^2 \to \Q \dL^2$ such that the graph of $\Phi_{F_\In}$ is an inertial manifold for the semigroup $\Psi(\cdot\,;\cdot,F_\In)$. 
    
    The choice of constants above ensures that all the $\{\Phi_{F_\In}, F_\In \in B_{r}\}$ are defined on the same spaces. We thus obtain a function $\Phi : \P \dL^2 \times B_{r} \to \Q \dL^2$, and according to~\cite[Ch.~VIII, Thm~3.2]{temam2012infinite}, the function is Lipschitz in $y\in \P \dL^2$, uniformly in $F_{\In}\in B_{r}$,  with Lipschitz constant $l=1/8$,
    \begin{equation*}
        |\Phi(y_1, F) - \Phi(y_2, F)| \leq l |y_1 - y_2|.
    \end{equation*}
    
    Furthermore, the function satisfies the following estimates: there exists $b \geq 0$, such that
    \begin{equation*}
        \sup_{y\in \P\dL^2, F_\In \in B_{r}}|\Phi(y,F_{\In})| \leq b,
    \end{equation*}

    and the support of $\Phi$ is bounded in $P\dL^2\times B_{r}$,
    \begin{equation*}
        \supp \Phi \subset \{(y,F_\In)\in P\dL^2 \times B_{r}, |y|< 2\rho(r) \}.
    \end{equation*}

    We have that $\mathcal{M}_{F} = \{(y+\Phi(y,F)) \colon y \in B(0,\rho(r))\}$ is an inertial manifold for $\Psi(\cdot\,;\cdot, F)$ given $F\in B_{r}$.

    Finally, we have the following representation of the function $\Phi(\cdot, F_{\In})$; for some localization function $\theta:\dR_+ \to [0,1]$, such that $\theta(x) = 1$ for $x\in [0,1]$,  and $\theta(x) = 0$ for $x>2$, with $\|\theta'\|_{\infty}\leq 2$, we set $\theta_{\rho}(x) = \theta(x/\rho(r))$ and $R_\theta (U, F) =  \theta_{\rho}(|U|)R(U,F)$. Then $\Phi(\cdot, F_{\In})$ admits the Lyapunov--Perron representation
    \begin{equation*}
        \Phi(y_0,F_{\In}) = \int_{-\infty}^0 e^{t A \Q}\Q (R_\theta(u,F_{\In}))\dd t,
    \end{equation*}
    where $u(t) = y(t) + \Phi(y(t),F_{\In})$ and $y \in C(\dR, \P\dL^2)$ is the unique solution of the ordinary differential equation
    \begin{equation}
        \label{eq:ODEInertialManifold}
        \frac{\dd}{\dd t} y + A y = \P (R_\theta(y + \Phi(y,F_{\In}),F_{\In})).
    \end{equation}
    Furthermore, we have the Lipschitz estimate
    \begin{equation}
        \label{est:LipschitLocalisationRtheta}
        \|A^{-\gamma}(R_\theta(U_1, F_1) - R_\theta(U_2, F_2))\|_{\dL^2} \leq M_2 \|U_1-U_2\|_{\dL^2} + c_\gamma \|F_1 - F_2 \|_{\dH^{-\vs}},
    \end{equation}
    for all $U_1, U_2 \in \P \dL^2$, $F_1, F_2 \in B_{r}$, where $c_\gamma = 2^{\gamma(n_{\max}-1)}(\tau_{\min}\wedge\nu_{\min})^{-\gamma}$, see~\cite[Ch.~VIII, Lem.~2.1]{temam2012infinite}.

    We only need to prove that $\Phi$ is Lipschitz in $F$ uniformly in $y$. To this end, take any $y_0\in \P\dL^2$, $F_1, F_2 \in B_{r}$, and consider $y_i (t) = y_i(t;y_0, \Phi(\cdot, F_i))$, the unique solution associated with the equation \eqref{eq:ODEInertialManifold}, denote $u_i = y_i + \Phi(y_i,F_i)$ and the difference $y = y_1 - y_2$. Then $y$ is a solution of the ODE
    \begin{equation*}
        \frac{\dd}{\dd t} y + A y = P(R_\theta(u_1,F_1) - R_\theta(u_2,F_2)).
    \end{equation*}
    
    Testing with $y$ gives
    \begin{align*}
        \tfrac{1}{2}\frac{\dd}{\dd t} |y|^2 + |A^{1/2} y|^2 = \langle A^{-\gamma} P(R_\theta(u_1,F_1) - R_\theta(u_2,F_2)), A^{\gamma} y \rangle.
    \end{align*}

    From the uniform Lipschitzness of $\Phi$, we have that
    \begin{equation*}
        |u_1 - u_2| \leq |y| + |\Phi(y_1, F_1) - \Phi(y_2, F_2)| \leq (1+l) |y| + |\osc_F (\Phi)|,
    \end{equation*}
    where $|\osc_F (\Phi)| = \sup_{y\in P\dL^2} |\Phi(y, F_1) - \Phi(y, F_2)|$. 
    
    Using the spectral results on $A$ from~\cite[Ch.~VIII, (3.5)--(3.6), p.~513]{temam2012infinite} and the Lipschitz estimate~\eqref{est:LipschitLocalisationRtheta}, we obtain the following inequality:
    \begin{equation*}
         \tfrac{1}{2}\frac{\dd}{\dd t} |y|^2 + \left|A^{1/2} y\right|^2 \geq -(M_2 ((1+l) |y|+ |\osc_F (\Phi)|) + c_\gamma \|\overline{F}\|_{\dH^{-\vs}}) |A^\gamma y|
    \end{equation*}
    where $\overline{F} = F_1 - F_2$, this leads to the estimate
    \begin{equation}
        \label{est:ODEInertialManifold}
        \frac{\dd}{\dd t} |y| \geq -\Tilde{\lambda} |y| - \lambda^\gamma(M_2|\osc_F (\Phi)|+ c_\gamma \|\overline{F}\|_{\dH^{-\vs}}),
    \end{equation}
    where $\Tilde{\lambda} := \lambda + M_2 (1+l) \lambda^\gamma $.

    We now use the representation of the function $\Phi$ as follows,

    \begin{align*}
        |\Phi(y_0,F_1) - \Phi(y_0, F_2)| &= \left|\int_{-\infty}^0 e^{t A Q}Q (R_\theta (u_1, F_1) - R_\theta(u_2, F_2))\dd t \right|,\\
        &\leq \left|\int_{-\infty}^0 e^{t A Q}Q (R_\theta (y_1 + \Phi(y_1, F_1), F_1) - R_\theta(y_2 + \Phi(y_2, F_1), F_1))\dd t \right|\\
        &\hspace{1em}+ \left|\int_{-\infty}^0 e^{t A Q}Q (R_\theta(y_2 + \Phi(y_2, F_1), F_1) - R_\theta(y_2 + \Phi(y_2, F_2), F_2))\dd t \right|\\
        & \hspace{1em} =: |\mathbf{I}_1|  +  |\mathbf{I}_2|,
    \end{align*}

    From the Lipschitz estimate of $R_\theta$, together with the integral estimate of~\cite[Ch.~VIII, Lem.~3.2]{temam2012infinite}, we obtain
    \begin{equation*}
        |\mathbf{I}_2| \leq \kappa \Lambda^{\gamma -1}(M_2|\osc_F (\Phi)|+ c_\gamma \|\overline{F}\|_{\dH^{-\vs}}).
    \end{equation*}
    
    For $\mathbf{I}_1$, applying~\cite[Ch.~VIII, Lem.~2.3]{temam2012infinite}, we can identify this integral as $\mathbf{I}_1 = \Tilde{z}(0)$, where $\Tilde{z}$ is the unique bounded global solution in $C(\dR, Q\dL^2)$ of 
    \begin{equation*}
        \frac{\dd}{\dd t} \Tilde{z} + A \Tilde{z} = Q (R_\theta (y_1 + \Phi(y_1, F_1), F_1) - R_\theta(y_2 + \Phi(y_2, F_1), F_1)).
    \end{equation*}
    From the Lipschitz estimate on $R_\theta$, we obtain
    \begin{equation*}
        \tfrac{1}{2}\frac{\dd}{\dd t}|\Tilde{z}|^2  + |A^{1/2}\Tilde{z}|^2 \leq M_2 (1+l)|y||A^\gamma\Tilde{z}|
    \end{equation*}
    and the spectral estimates lead to
    \begin{equation*}
        \frac{\dd}{\dd t}|\Tilde{z}| \leq - \Lambda |\Tilde{z}| + M_2 \Lambda^\gamma(1+l) |y|.
    \end{equation*}

    When $|\Tilde{z}|>|y|$ we have

    \begin{equation}
        \label{est:ODEInertialManifold2}
        \frac{\dd}{\dd t}|\Tilde{z}| + \Tilde{\Lambda}|\Tilde{z}| \leq 0,
    \end{equation}
    with $\Tilde{\Lambda} = \Lambda - M_2 \Lambda^\gamma(1+l)$, and $\Tilde{\Lambda}>0$ from assumption~\eqref{hyp:InertialManifoldSpecGap}.
    Recall that, by definition, $y(0) = y_1(0) - y_2(0) = 0$. If $\Tilde{z}(0) = 0$, then $|\mathbf{I}_1| = 0$ and the property is proved. If $\Tilde{z}(0) \neq 0$, then, by continuity, we have that $|\Tilde{z}(t)|>|y(t)|$, for $t<0$, with $|t|$ sufficiently small. We then have two cases, either $|\Tilde{z}(t)|>|y(t)|$ for all $t<0$ and from Gr\"onwall's lemma on estimate~\eqref{est:ODEInertialManifold2}, we obtain
    \begin{equation*}
        |\Tilde{z}(0)| \leq |\Tilde{z}(t)|e^{\Tilde{\Lambda}t}, \forall t <0
    \end{equation*}
    and as $\Tilde{z}$ is globally bounded (see~\cite[Ch.~VIII, Lem.~2.3]{temam2012infinite}) we obtain that, $|\Tilde{z}(0)| = 0$, leading to a contradiction. In the second case, we obtain that $|\Tilde{z}(t)| > |y(t)|$ on $(t_0,0]$, and $|\Tilde{z}(t_0)| = |y(t_0)|$. Subtracting estimate~\eqref{est:ODEInertialManifold} from estimate~\eqref{est:ODEInertialManifold2}, we obtain the inequality

    \begin{align*}
        \frac{\dd}{\dd t}(|\Tilde{z}| - |y|) + (\Tilde{\Lambda} - \Tilde{\lambda}) (|\Tilde{z}| - |y|) & \leq (\Tilde{\Lambda} - \Tilde{\lambda}) (|\Tilde{z}| - |y|) - \Tilde{\Lambda}|\Tilde{z}| + \Tilde{\lambda}|y| + \lambda^\gamma(M_2|\osc_F (\Phi)|+ c_\gamma \|\overline{F}\|_{\dH^{-\vs}}),\\
        &\leq - \Tilde{\lambda} (|\Tilde{z}| - |y|) +\lambda^\gamma(M_2|\osc_F (\Phi)|+ c_\gamma \|\overline{F}\|_{\dH^{-\vs}}),\\
        &\leq \lambda^\gamma(M_2|\osc_F (\Phi)|+ c_\gamma \|\overline{F}\|_{\dH^{-\vs}}).
    \end{align*}

    Applying the Gr\"onwall lemma on the above between $t_0$ and $0$, and using that at $t_0$, $|\Tilde{z}(t_0)| - |y(t_0)| = 0$, and that $|y(0)| = 0$, we obtain
    \begin{equation*}
        |\mathbf{I}_1| = |\Tilde{z}(0)| \leq \int^0_{t_0} e^{t(\Tilde{\Lambda}-\Tilde{\lambda})}\dd t \lambda^\gamma(M_2|\osc_F (\Phi)|+ c_\gamma \|\overline{F}\|_{\dH^{-\vs}}) \leq \frac{1}{\Tilde{\Lambda}-\Tilde{\lambda}}\lambda^\gamma(M_2|\osc_F (\Phi)|+ c_\gamma \|\overline{F}\|_{\dH^{-\vs}}).
    \end{equation*}
    Summing up the estimates on $|\mathbf{I}_1|$ and $|\mathbf{I}_2|$ we obtain
    \begin{equation*}
        |\Phi(y_0,F_1) - \Phi(y_0, F_2)| \leq \eta |\osc_F (\Phi)| + \eta c_\gamma \|\overline{F}\|_{\dH^{-\vs}}
    \end{equation*}
    where $\eta = M_2 (\kappa \Lambda^{\gamma -1} + \frac{1}{\Tilde{\Lambda}-\Tilde{\lambda}} \lambda^\gamma) $, and from assumptions~\eqref{hyp:InertialManifoldSpecGrowth}--\eqref{hyp:InertialManifoldSpecGap} and as $M_2 \geq 1$, the above holds with $\eta \leq 1/2$. Taking the supremum in $y_0$ and rearranging the terms, we obtain
    \begin{equation*}
        \|\Phi(y_0,F_1) - \Phi(y_0, F_2)\|_{\dL^2} \leq c_\gamma \|F_1 - F_2\|_{\dH^{-\vs}}, \text{ for all } y_0 \in P\dL^2, F_1, F_2 \in B_{r}.
    \end{equation*}
    This establishes the existence of a Lipschitz function $\Phi \colon P\dL^2 \times B_r \to Q \dL^2$ whose graph is an inertial manifold for the equation~\eqref{eq:AbstractModelFin}. The dimension of the inertial manifold for the equation~\eqref{eq:AbstractModel} is the dimension of the domain of $\Phi$. The upper bound follows from
    \begin{equation*}
        \dim \mathcal{M} = \dim \P \dL^2  + \dim \W^\In V = N + \Rank \W^\In
    \end{equation*}
    and the estimate~\eqref{est:DimensionalEstimateinProof}.
\end{proof}